\documentclass[11pt]{article}
\usepackage[T1]{fontenc}
\usepackage[latin1]{inputenc}
\usepackage[margin=0.8in]{geometry}
\usepackage{amsmath,amssymb,amstext}
\usepackage{amsthm}
\usepackage{cleveref}
\usepackage{authblk}
\usepackage{enumerate}
\usepackage{dirtytalk}
\usepackage{tikz}
\usepackage{thmtools}
\usepackage{thm-restate}
\usetikzlibrary{shapes}

\newtheorem{thm}{Theorem}[section]

\newtheorem{prop}[thm]{Proposition}
\newtheorem{lemma}[thm]{Lemma}
\newtheorem{conj}[thm]{Conjecture}
\newtheorem{cor}[thm]{Corollary}
\newtheorem{claim}{Claim}

\newtheorem{obs}[thm]{Observation}

\numberwithin{equation}{section}
\theoremstyle{definition} 
\newtheorem{definition}[thm]{Definition}

\newcommand{\defc}{\textnormal{def}}

\definecolor{asparagus}{rgb}{0.53, 0.66, 0.42}

\definecolor{cerulean}{rgb}{0.0, 0.48, 0.65}
\definecolor{cornellred}{rgb}{0.7, 0.11, 0.11}
\definecolor{darklavender}{rgb}{0.45, 0.31, 0.59}
\definecolor{darkslateblue}{rgb}{0.28, 0.24, 0.55}
\definecolor{burntorange}{rgb}{0.8, 0.33, 0.0}
\date{}

\begin{document}


\author{Luke Postle\thanks{We acknowledge the support of the Natural Sciences and Engineering Research Council of Canada (NSERC) [Discovery Grant No.  2019-04304].\\
\hphantom{m} $^*$Cette recherche a \'{e}t\'{e} financ\'{e}e par le Conseil de recherches en sciences naturelles et en g\'{e}nie du Canada (CRSNG)[Discovery Grant No.  2019-04304].} }
\author{Evelyne Smith-Roberge$^\dagger$}
\affil{$^*$Dept.~of Combinatorics and Optimization, University of Waterloo \\ \texttt{lpostle@uwaterloo.ca}}
\affil{$^\dagger$School of Mathematics, Georgia Institute of Technology \\ \texttt{esmithroberge3@gatech.edu}}
\title{Exponentially Many Correspondence Colourings of Planar and Locally Planar Graphs}
\date{\today}
\maketitle
\date{}

\begin{abstract}
We show that there exists a constant $c > 0$ such that if $G$ is a planar graph with 5-correspondence assignment $(L,M)$, then $G$ has at least $2^{c\cdot v(G)}$ distinct $(L,M)$-colourings. This confirms a conjecture of Langhede and Thomassen. More broadly, we introduce a general method showing how hyperbolicity theorems for certain families of critical graphs can be used to derive lower bounds on the number of colourings of the associated class of planar graphs. Hence our main result follows from this method plus a technical theorem (that we proved in a previous paper) involving the hyperbolicity of graphs critical for $5$-correspondence colouring. We further demonstrate our method in the case of counting 3-correspondence colourings of planar graphs of girth at least five. Finally, we use these theorems to show analogous results hold in the case of counting 5-correspondence colourings of locally planar graphs, and counting 3-correspondence colourings of locally planar graphs of girth at least five.
\end{abstract}
\maketitle


\section{Introduction}


\subsection{Results}\label{subsec:results}
All graphs in this paper are simple and finite. Given a graph $G$ with vertex-set $V(G)$ and edge-set $E(G)$, we denote $|V(G)|$ and $|E(G)|$ by $v(G)$ and $e(G)$, respectively.  A \emph{$k$-colouring} of $G$ is a function $\varphi:V(G) \rightarrow \{1, 2, \dots, k\}$ with the property that $\varphi(u) \neq \varphi(v)$ for each $uv \in E(G)$. We say $G$ is \emph{$k$-colourable} if there exists a $k$-colouring of $G$.  A \emph{$k$-list assignment} for a graph $G$ is a function $L$ with domain $V(G)$ that assigns to each vertex $v \in V(G)$ a set $L(v)$ (called a \emph{list of colours}) with $|L(v)| \geq k$. We say $G$ is \emph{$L$-colourable} if there exists a colouring $\varphi$ of $G$ with $\varphi(v) \in L(v)$ for each $v \in V(G)$. We say $G$ is \emph{$k$-list-colourable} (or \emph{$k$-choosable}) if there exists an $L'$-colouring for every $k$-list assignment $L'$ of $G$.

Many questions in the field of graph colouring involve determining whether or not a graph with specific structure is colourable (for some notion of colouring). Relatedly, we might ask the following: given that a graph \emph{is} colourable, how easy is it to find a colouring? 

One way to answer this question is to investigate how many distinct colourings of the graph there are. This sort of question has already been studied extensively for list and ordinary colourings: for instance, Birkhoff and Lewis \cite{birkhoff1946chromatic} showed that if $G$ is a planar graph, then $G$ has at least $60 \cdot 2^{v(G)-3}$ distinct 5-colourings. Similar results were shown for list colouring; to present these results concisely, we introduce the following definition. 

\begin{definition}\label{def:expmany}
Let $\mathcal{G}$ be a class of graphs. We say $\mathcal{G}$ has \emph{exponentially many $k$-list colourings} if there exists a constant $c > 0$ such that for every graph $G \in \mathcal{G}$ and every $k$-list assignment $L$ of $G$, there exist at least $2^{c\cdot v(G)}$ distinct $L$-colourings of $G$. 
\end{definition}

Thomassen  \cite{thomassen5LC} famously showed in 1994 that planar graphs are 5-list colourable. In 1993, Voigt \cite{voigt1993list} gave a construction of a planar graph that is not 4-list colourable: thus Thomassen's result is best possible. In 2007, Thomassen \cite{thomassen2007exponentially} proved that planar graphs have exponentially many 5-list colourings (with $c = \frac{1}{9}$). Recently in 2022, it was observed by Bosek, Grytczuk, Gutowski, Serra, and Zaj\k{a}c \cite{bosek2022graph} that, using the polynomial method (discussed further in Subsection \ref{subsec:polmethod}), it is straightforward to improve this bound (to $c = \frac{\log_2 5}{4}$). 

Analogous results hold for planar graphs of higher girth and with smaller list sizes. Recall that the \emph{girth} of a graph is the length of a shortest cycle in the graph (where if the graph is a forest, we define the girth to be infinite). In 1995, Thomassen proved that planar graphs of girth at least five are 3-list colourable \cite{thomassen3LC}, and later in 2007 \cite{thomassen2007many}, that this class of graphs has exponentially many 3-list colourings (with $c = \frac{1}{10000}$). Again, it was observed by Bosek et al. \cite{bosek2022graph} that the bound can be further improved (in this case, to $c = \frac{\log_2 3}{6}$). 

One might wonder whether analogous results hold for other forms of colouring: in particular, whether these results extend to \emph{correspondence colouring.} Correspondence colouring is a natural generalization of list colouring first introduced by Dvo{\v{r}}{\'a}k and the first author in 2015 \cite{dvovrak2018correspondence}. It is defined as follows.
\begin{definition}
 Let $G$ be a graph. A \emph{$k$-correspondence assignment for $G$} is a $k$-list assignment $L$ together with a function $M$ that assigns to every edge $e = uv \in E(G)$ a partial matching $M_e$ between $\{u\}\times L(u)$ and $\{v\}\times L(v)$.
An $(L,M)$-colouring of $G$ is a function $\varphi$ that assigns to each vertex $v \in V(G)$ a colour $\varphi(v) \in L(v)$ such that for every $e = uv \in E(G)$, the vertices $(u, \varphi(u))$ and $(v, \varphi(v))$ are non-adjacent in $M_e$. We say that $G$ is $(L,M)$-colourable if such a colouring exists, and that $G$ is \emph{$k$-correspondence-colourable} if $G$ is $(L,M)$-colourable for every $k$-correspondence assignment $(L,M)$ for $G$.
\end{definition}

Since its introduction, there have been over a hundred papers on correspondence colouring; see for example \cite{bernshteyn2016asymptotic,bernshteyn2018sharp,bernshteyn2019differences,bernshteyn2017dp,liu2019dp,zhang2021edge} for a small sample. In the paper in which they introduce correspondence colouring, Dvo{\v{r}}{\'a}k and the first author observed that Thomassen's proofs of the 5-choosability of planar graphs and 3-choosability of planar graphs of girth at least five naturally carry over to the correspondence colouring framework. That is: planar graphs are 5-correspondence colourable, and planar graphs of girth at least five are 3-correspondence colourable. In light of this, it is natural to wonder whether there are exponentially many 5-correspondence colourings of planar graphs. Hence we have the following definition.

\begin{definition}\label{def:expmany2}
Let $\mathcal{G}$ be a class of graphs. We say $\mathcal{G}$ has \emph{exponentially many $k$-correspondence colourings} if there exists a constant $c > 0$ such that for every graph $G \in \mathcal{G}$ and every $k$-correspondence assignment $(L,M)$ of $G$, there exist at least $2^{c\cdot v(G)}$ distinct $(L,M)$-colourings of $G$. 
\end{definition}

Indeed in 2021, Langhede and Thomassen~\cite{langhede2021exponentially} conjectured as follows.

\begin{conj}[Langhede and Thomassen, \cite{langhede2021exponentially}]\label{thomconj}
Planar graphs have exponentially many 5-correspondence colourings.
\end{conj}

In \cite{langhede2021exponentially}, Langhede and Thomassen proved that planar graphs have exponentially many $\mathbb{Z}_5$-colourings; their bound was improved and the result was strengthened to the more general framework of field colourings by Bosek et al. in \cite{bosek2022graph}.  Field colouring generalizes ordinary colouring, but is not a generalization of list colouring; correspondence colouring, however, is a generalization of both list colouring and field colouring. It follows that a proof of Conjecture \ref{thomconj} would imply Langhede and Thomassen's result on exponentially many $\mathbb{Z}_5$-colourings as well as Bosek et al.'s result on field colourings (albeit with perhaps a worse value of the constant $c$ in both cases). 
One of our main results is a proof of Conjecture~\ref{thomconj}. In particular, we prove the following.

\begin{restatable}{thm}{expmanycc}\label{expmanycc}
If $G$ is a planar graph and $(L,M)$ is a 5-correspondence assignment for $G$, then $G$ has at least $2^\frac{v(G)}{67}$ distinct $(L,M)$-colourings.
\end{restatable}

One might wonder whether similar results hold for graphs embedded in other surfaces; in particular, for \emph{locally planar} graphs.

\begin{definition}
A \emph{non-contractible cycle} in a surface is a cycle that cannot be continuously deformed to a single point. An embedded graph is \emph{$\rho$-locally planar} if every cycle (in the graph) that is non-contractible (in the surface) has length at least $\rho$.
\end{definition}
This is closely related to the concept of \emph{edge-width}. The edge-width of an embedded graph is the length of the shortest non-contractible cycle; if a graph is $\rho$-locally planar, it has edge-width at least $\rho$.

As the name suggests, the local structure of a locally planar graph is planar. Pleasantly, locally planar graphs sometimes share colouring properties with planar graphs: in 2006, DeVos, Kawarabayashi, and Mohar \cite{devos2006locally} showed that for every surface $\Sigma$, there exists a constant $\rho= 2^{O(g)}$, where $g$ is the Euler genus of $\Sigma$, such that every $\rho$-locally planar graph that embeds in $\Sigma$ is 5-list-colourable. Thomassen proved a similar result for 5-colourability in 1993 \cite{thomassen1993five}. In \cite{postle2018hyperbolic}, Postle and Thomas showed that for $k \in \{3,4,5\}$, analogous results for $(8-k)$-list-colouring graphs of girth at least $k$ (with $\rho = \Omega(\log(g))$) are implied by the \emph{hyperbolicity} of certain associated families of graphs. \emph{Hyperbolicity} is defined below. Throughout the rest of the paper, we use $(G, \Sigma)$ to denote a graph $G$ embedded in a surface $\Sigma$.
\begin{definition}\label{def:hyp}
Let $\mathcal{F}$ be a family of embedded graphs. We say that $\mathcal{F}$ is \emph{hyperbolic} if there exists a constant $c > 0$ such that if $(G, \Sigma) \in \mathcal{F}$, then for every closed curve $\eta : S^1 \rightarrow \Sigma$ that bounds an open disk $\Delta$ and intersects $G$ only in vertices, if $\Delta$ includes a vertex of $G$, then the number of vertices of $G$ in $\Delta$ is at most $c(|\{x \in S^1 : \eta(x) \in V (G)\}| - 1)$. We say that $c$ is a \emph{Cheeger constant} for $\mathcal{F}$.
\end{definition}
Note that in the above definition, each vertex $v\in V(G)$ contributes 1 to the quantity $|\{x \in S^1 : \eta(x) \in V (G)\}|$ for each time $\eta$ intersects $v$.

Thus for each $k \in \{3,4,5\}$, if $(G,\Sigma)$ is a locally planar embedded graph of girth at least $(8-k)$, then $G$ has exponentially many distinct $L$-colourings for every $k$-list assignment $L$.  In the correspondence colouring framework, much less is known. It was shown by Kim, Kostochka, Li, and Zhu in \cite{kim2020line} and independently by the authors in \cite{esrlukelocal} that locally planar graphs are 5-correspondence colourable; and the authors also observe in \cite{esrlukelocal} that locally planar graphs of girth at least five are 3-correspondence colourable. Our results, like those of Postle and Thomas, follow from the hyperbolicity of certain associated graph families. In Section \ref{sec:locallyplanar}, we prove the following result, showing that not only is it true that locally planar graphs are $5$-correspondence colourable, but in fact they have exponentially many $5$-correspondence colourings.

\begin{restatable}{thm}{thmexpmanyloc}\label{thm:expmanylocallyplanar}
For every surface $\Sigma$, there exists a constant $\rho > 0$ with $\rho = O(\log g)$ (where $g$ is the Euler genus of $\Sigma$) such that the following holds: if $G$ is a $\rho$-locally planar graph that embeds in $\Sigma$ and $(L,M)$ is a 5-correspondence assignment for $G$, then $G$ has at least $2^\frac{v(G)}{3484}$ distinct $(L,M)$-colourings. 
\end{restatable}

We also give an analogous theorem for 3-correspondence colouring locally planar graphs of girth at least five.

Our theorems also follow from the hyperbolicity of certain graph families: thus hyperbolicity theorems can be used to show not only that planar and locally planar graphs are colourable (for some notion of colouring), but also that there exist exponentially many colourings of both planar and locally planar graphs. That exponentially many colourings follows from these hyperbolicity theorems (rather than the related notion of \emph{strong} hyperbolicity, which bounds the number of vertices in annuli rather than disks) was not known. For more interesting implications of hyperbolicity, we refer the reader to \cite{postle2018hyperbolic}.


\subsection{From Hyperbolicity to Exponentially Many Colourings: A General Method}\label{subsec:hyperbolicity}

Another contribution of this paper is a general method to show there exist exponentially many colourings (for some notion of colouring) of graphs on surfaces. Namely our method shows that deficiency versions of the hyperbolicity of the associated family of critical graphs implies exponentially many colourings of planar graphs in the class (and following the work of the first author and Thomas~\cite{postle2018hyperbolic} also of graphs in the class embedded in a fixed surface provided the class is also strongly hyperbolic).

The theory of hyperbolic families of graphs was developed by the first author and Thomas in \cite{postle2018hyperbolic}. Loosely speaking, a family of embedded graphs is \emph{hyperbolic} if every graph in the family satisfies a certain linear isoperimetric inequality (given in the definition above): roughly, that for every disk in the surface, the number of vertices in the disk is linear in the number of vertices in the boundary of the disk. Similarly, a family is \emph{strongly hyperbolic} if the number of vertices in each annulus is a linear function of the number of vertices in its two boundary components. This is defined more formally in \cite{postle2018hyperbolic}.

Per \cite{postle2018hyperbolic}, many results involving colouring graphs on surfaces are implied by the hyperbolicity of the associated class of \emph{critical} graphs (i.e.~the minimal non-colourable graphs in the class).  Pleasantly, to prove hyperbolicity, it usually suffices to prove a related result for extending a precolouring of the outer cycle of a planar graph (or two cycles, in the case of strong hyperbolicity). Hence many theorems about colouring graphs on surfaces can be reduced to theorems about planar graphs via the theory of hyperbolicity. 

The implications of such hyperbolicity-related theorems are discussed in more detail in \cite{lukeevehyperbolicity} and \cite{postle2018hyperbolic}; they include among others:~colouring theorems for locally planar (embedded) graphs, and efficient algorithms for colouring decidability of graphs embedded in fixed surfaces (per the work of Dvo{\v{r}}{\'a}k and Kawarabayashi \cite{dvovrak2013list}). Thus another main contribution of this paper is to show that such theorems can also be used to prove the existence of exponentially many colourings of planar graphs in the class. Technically, we need a \emph{deficiency} version of a hyperbolicity theorem (one that bounds the number of vertices not only in the number of boundary vertices of the disk but also the number of edges inside the disk; see Section~\ref{sec:keyresults} for a formal definition). That our results follow from these deficiency hyperbolicity theorems further motivates the study of hyperbolic families.

Hence our Theorem~\ref{expmanycc} will follow from the deficiency hyperbolicity theorem for $5$-correspondence coloring that we proved in a previous paper~\cite{lukeevehyperbolicity}. Our method can also be applied in other colouring settings where the appropriate deficiency hyperbolicity theorem is known: as another demonstration of our method, we show in Section \ref{sec:girth5} that our technique can also be used to prove there exist exponentially many 3-correspondence colourings of planar graphs of girth at least five. 

In particular, we show the following.
\begin{restatable}{thm}{expmanyccgirthfive}\label{expmanyccgirthfive}
If $G$ is a planar graph and girth at least five and $(L,M)$ is a 3-correspondence assignment for $G$, then $G$ has at least $2^\frac{v(G)}{282}$ distinct $(L,M)$-colourings.
\end{restatable}

We note that the bound above (which originally appeared in the PhD thesis of the second author \cite{evethesis}) was recently improved by Dahlberg, Kaul, and Mudrock \cite{dahlberg2023algebraic} (to $c = \frac{\log_2 3}{6}$) using the polynomial method. We discuss the polynomial method further in the next subsection.

Analogously to the 5-correspondence colouring case, we further show Theorem \ref{expmanyccgirthfive} can be used to prove the theorem below.
\begin{restatable}{thm}{thmexpmanylocfive}\label{thm:expmanylocallyplanarg5}
For every surface $\Sigma$, there exists a constant $\rho > 0$ with $\rho = O(\log g)$ (where $g$ is the Euler genus of $\Sigma$) such that the following holds: if $G$ is a $\rho$-locally planar graph of girth at least five that embeds in $\Sigma$ and $(L,M)$ is a 3-correspondence assignment for $G$, then $G$ has at least $2^\frac{v(G)}{25380}$ distinct $(L,M)$-colourings. 
\end{restatable}
\subsection{Comparison to Polynomial Method}\label{subsec:polmethod}

We should remark that the proofs of the hyperbolicity theorems mentioned in the foregoing subsection are often long and technical. This might motivate the skeptical reader to wonder why one should use our method (which relies on the existence of the hyperbolicity theorems) at all, given the existence of (for instance) the polynomial method, which has been used to tackle similar questions involving counting graph colourings and often yields better bounds. Our reasoning is twofold: first, as alluded to above and explained in more detail in \cite{postle2018hyperbolic}, hyperbolicity theorems come with a host of other interesting implications. For that reason, it is of independent interest to prove them; and once they are proved, the remainders of the proofs of Theorems \ref{expmanycc} and \ref{expmanyccgirthfive} are quite concise (approximately two pages for each proof). Our method demonstrates yet another interesting implication of these hyperbolicity theorems, and in so doing further motivates their study. 

Second, our method can be applied to more colouring types and graph class combinations than the polynomial method (for instance, 5-correspondence colouring planar graphs). As mentioned in Subsections \ref{subsec:results} and \ref{subsec:hyperbolicity}, the polynomial method has been used to attain the best known bounds on $c$ for showing there are exponentially many 5-list colouring planar graphs \cite{bosek2022graph}, exponentially many 3-list colouring planar graphs of girth at least five \cite{bosek2022graph}, and exponentially many 3-correspondence colourings of planar graphs of girth at least five \cite{dahlberg2023algebraic}.

We give a brief summary of the polynomial method. In \cite{bosek2022graph} and \cite{dahlberg2023algebraic}, the authors use a slightly weaker version of a famous theorem of Alon and F{\"u}redi \cite{alon1993covering} concerning the number of non-vanishing points of a polynomial $P(x_1, x_2, \dots, x_n)$ over a set $B \subseteq \mathbb{F}^n$, where $\mathbb{F}$ is a field. In particular, they use the following.

\begin{thm}[\cite{bosek2022graph}]\label{thm:alonfuredi}
Let $\mathbb{F}$ be an arbitrary field, and let $A_1,A_2,\dots, A_n$ be any nonempty subsets of $\mathbb{F}$ with $S = \sum_{i=1}^n |A_i|$ and $t = \max |A_i|$. Let $B = A_1 \times A_2 \times \dots \times A_n$, and suppose that $P(x_1, \dots, x_n)$ is a polynomial over $\mathbb{F}$ of degree $d$ that does not vanish on all of $B$. Then the number of points in $B$ for which $P$ has a non-zero value is at least $t^{\frac{S-n-d}{t-1}}$ provided $t \geq 2$. 
\end{thm}
This theorem is the main tool. The general method is to then associate a graph $G$ with vertex set $\{v_1, \dots, v_n\}$ and list assignment $L(v_i) = A_i$ for all $i \in \{1, \dots, n\}$ a polynomial $P(x_1, \dots, x_n) = \prod_{v_iv_j \in E(G) } (x_i-x_j)$.  For a point $(a_1, \dots, a_n) \in A_1 \times \dots \times A_n$, we have that $P(a_1, \dots, a_n)$ is non-vanishing if and only if $a_i \neq a_j$ for all $v_iv_j \in E(G)$. Since  $a_i \in A_i = L(v_i)$ for all $i \in \{1, \dots, n\}$, the point $(a_1, \dots, a_n)$ corresponds to a valid $L$-colouring of $G$. If $G$ is planar, the value of $d$ (which corresponds to $e(G)$) is appropriately bounded and it then follows from Theorem \ref{thm:alonfuredi} that there are exponentially many $L$-colourings.

In \cite{dahlberg2023algebraic}, the authors define a different polynomial to that described above; their polynomial encodes 3-correspondence colouring, rather than the more straightforward 3-list colouring. For a graph $G$ with correspondence assignment $(L,M)$, we may assume without loss of generality that every vertex $v$ has $L(v) = \mathbb{F}_3$, the finite field with three elements. Recall that in the case of correspondence colouring, the potential conflicts between vertices $u,v$ with $uv \in E(G)$ are described by a partial matching $M_{uv}$ between $L(v)$ and $L(u)$. Since $G$ has the fewest possible $(L,M)$-colourings when $M_{uv}$ is a perfect matching for all $uv \in E(G)$, it follows that when lower-bounding the number of $(L,M)$-colourings of $G$ we may assume that $M_{uv}$ describes a permutation $\sigma$ of $\mathbb{F}_3$. The authors observe that for each such permutation $\sigma$, either $z -\sigma(z)$ is the same for all $z \in \mathbb{F}_3$, or $z + \sigma(z)$ is the same for all $z \in \mathbb{F}_3$. This motivates defining the factor of the polynomial $P$ associated with $v_iv_j \in E(G)$ as $(x_i+(-1)^c x_j-a)$, where $c$ and $a$ are chosen so that this factor is zero if and only if the values of $x_i$ and $x_j$ are matched in $M_{v_iv_j}$. However, this strategy does not work in the case of 5-correspondence colouring, as no such observation concerning $\sigma$ (and resulting in linear factors for each of the edges in the graph) holds.  Though the polynomial method used by Dahlberg, Kaul, and Mudrock \cite{dahlberg2023algebraic} and Bosek et al. \cite{bosek2022graph} gives better explicit bounds on the number of colourings than our method or the more ad-hoc approaches of Thomassen \cite{thomassen2007exponentially, thomassen2007many}, our method does not have the same algebraic constraints, and thus applies to more diverse notions of colouring.

In addition, as we show in Section \ref{sec:locallyplanar}, we can use our hyperbolicity results to extend  these theorems counting planar colourings to counting colourings of locally planar graphs.


\subsection{Outline of Paper}\label{subsec:paperoutline}
In Subsection \ref{subsec:methodoverview}, we give a brief overview of our method and the proofs of Theorems \ref{expmanycc} and \ref{expmanyccgirthfive}. Section \ref{sec:keyresults} contains several key results (some from previous papers) which will be used in the remainder of the paper. In particular, Subsection \ref{subsec:deficiencytheorems} contains theorems related to deficiency and hyperbolicity, and Subsection \ref{subsec:corr-delsub} introduces \emph{correspondence-deletable subgraphs} and uses the results from Subsection \ref{subsec:deficiencytheorems} to derive further useful tools for the proof of Theorem \ref{expmanycc}. Section \ref{sec:girth3} contains a proof of Theorem \ref{expmanycc}. The tools required for the proof of Theorem \ref{expmanyccgirthfive} are found in Subsection \ref{subsec:toolsgirth5}; the proof of Theorem \ref{expmanyccgirthfive} is in Subsection \ref{subsec:proofsgirth5}. Section \ref{sec:locallyplanar} contains the proofs of Theorems \ref{thm:expmanylocallyplanar} and \ref{thm:expmanylocallyplanarg5}. Finally, Section \ref{sec:conclusion} contains a discussion on further directions.

\subsection{Proof and Method Overview}\label{subsec:methodoverview}

Our main theorem is Theorem \ref{expmanycc} (proved in Section \ref{sec:girth3}). In fact, we prove a more technical version of Theorem \ref{expmanycc}\textemdash Theorem \ref{expmanyextensions}\textemdash which involves counting the number of extensions of a precoloured connected subgraph $S$ of a planar graph $G$ to a 5-correspondence colouring of $G$ itself. We show that if the precolouring of $S$ has at least one extension, then it has exponentially many extensions. The precise number of extensions is counted in terms of the deficiency of $G$, defined more formally in Section \ref{sec:keyresults}. This stronger, deficiency version of Theorem \ref{expmanycc} is more amenable to an inductive proof, as will be clear to the reader in Section \ref{sec:girth3}. Analogously, in Section \ref{sec:girth5}, we prove Theorem \ref{expmanyccgirthfive} via Theorem \ref{expmanyextensions5}, which concerns extending a precolouring of a connected subgraph $S$ of a planar graph $G$ of girth at least five to a 3-correspondence colouring of $G$. 

Both the proof of Theorem \ref{expmanyextensions} and that of Theorem \ref{expmanyextensions5} proceed by induction on $v(G) - v(S)$. There are essentially two cases to consider: either there exists a subgraph $H$ with $S \subsetneq H \subsetneq G$ and $v(S)<v(H)<v(G)$ such that every colouring of $H$ extends to a colouring of $G$, or no such subgraph $H$ exists. In the former case, by our choice of $G$, since $v(S)<v(H)<v(G)$ and $S \subsetneq H \subsetneq G$, we have that both that $v(H)-v(S) < v(G) -v(S)$ and that $v(G)-v(H)< v(G)-v(S)$. Theorem \ref{expmanyextensions} then follows by induction by first counting the extensions of the colouring of $S$ to $H$, and then of each colouring of $H$ to $G$. 

In the latter case where no such subgraph $H$ exists, we use our deficiency hyperbolicity theorems (Theorems \ref{theorem:stronglinear} and \ref{girth5:stronglinear}) to show that $G$ has relatively high deficiency with respect to $S$ (see Theorem \ref{d_gepsilonbound} in the case of 5-correspondence colouring and Theorem \ref{d_gepsilonboundg5} in the case of 3-correspondence colouring).  We then pick a vertex $v$ outside $S$ with the most neighbours in $S$, and then consider extending the precolouring of $S$ first to $S+v$ and then to $G$. 
Since $S + v$ does not have the properties of a graph $H$ as described in the previous paragraph, we use the deficiency theorems to bound the deficiency of the whole graph, at which point the formula follows since there is at least one extension by assumption. Other key tools are Theorems  \ref{thomtech5cc} and \ref{thomtech3cc} (in the girth three and five case, respectively), which are Thomassen's stronger inductive theorems that imply that planar graphs are 5-correspondence colourable, and that planar graphs of girth at least five are 3-correspondence colourable. These help bound the deficiency in the final steps of the proofs. We note the case where $V(G) = V(S) \cup \{v\}$ is especially important, as it helps determine the coefficient in the exponent of Theorems \ref{expmanycc} and \ref{expmanyccgirthfive}. 

The general method then may be described as follows: the first step is to define the correct notion of deficiency and deletability for the graph class and type of colouring with which we are working. We then prove a deficiency version of a hyperbolicity theorem for the critical graphs in the class under study, and use this to derive a bound on the deficiency of graphs that do not contain deletable subgraphs (see Definition \ref{def:deletable}). One final key tool is a theorem analogous to Theorem \ref{thomtech5cc}.


\section{Our Toolbox}\label{sec:keyresults}

In this section, we develop the tools required for the proof of Theorem \ref{expmanycc}. In particular, Subsection \ref{subsec:deficiencytheorems} contains our hyperbolicity theorem for 5-correspondence colouring and its stronger, deficiency version from our previous paper \cite{lukeevehyperbolicity}. This will be our main tool in proving Theorem \ref{expmanycc}. Subsection \ref{subsec:corr-delsub} introduces the notion of \emph{correspondence-deletable subgraphs}, and derives from Theorem \ref{theorem:stronglinear} (our deficiency hyperbolicity theorem for 5-correspondence colourings) useful results concerning said deletable subgraphs.  The analogous tools for 3-correspondence colouring used in the proof of Theorem \ref{expmanyccgirthfive} are found in Section \ref{sec:girth5} (Subsection \ref{subsec:toolsgirth5}). 

\subsection{Deficiency Theorems for Hyperbolicity}\label{subsec:deficiencytheorems}


This subsection introduces our main tool, Theorem \ref{theorem:stronglinear} (from one of our previous papers \cite{lukeevehyperbolicity}). Theorem \ref{theorem:stronglinear} concerns the hyperbolicity of a specific class of \emph{critical} graphs. We recall the following definition.
\begin{definition}
Let $G$ be a planar graph, $S$ a proper subgraph of $G$, and $(L,M)$ a correspondence assignment for $G$. We say $G$ is \emph{$S$-critical with respect to $(L,M)$} if for every proper subgraph $G' \subset G$ such that $S \subseteq G'$, there exists an $(L,M)$-colouring of $S$ that extends to an $(L,M)$-colouring of $G'$, but does not extend to an $(L,M)$-colouring of $G$. If the correspondence assignment is clear from the context, we shorten this and say that $G$ is \emph{$S$-critical.}
\end{definition}

The main hyperbolicity theorem presented in \cite{lukeevehyperbolicity} is as follows.
\begin{thm}\cite{lukeevehyperbolicity}\label{theorem:mainhypthm}
Let $\varepsilon = \frac{1}{50}$, let $G$ be a 2-connected plane graph, let $C$ be the outer cycle of $G$, and let $(L,M)$ be a 5-correspondence assignment for $G$. If $G$ is $C$-critical with respect to $(L,M)$, then $v(G) \leq \frac{1+\varepsilon}{\varepsilon}\cdot v(C)$.
\end{thm}

Theorem \ref{theorem:mainhypthm} is functionally equivalent to proving the hyperbolicity of the family of graphs that are $\emptyset$-critical for some $5$-correspondence assignment and hence can be used to obtain the various applications of hyperbolicity to described in the introduction to $5$-correspondence colouring. Further discussion on this topic can be found in \cite{lukeevehyperbolicity}.


However, Theorem \ref{theorem:mainhypthm} is not quite strong enough to obtain our result lower-bounding the number of 5-correspondence colourings of planar graphs. To prove Theorem \ref{expmanycc}, we need a stronger theorem which bounds $v(G)$ not only in terms of $v(C)$, but also in terms of $e(G)-e(C)$. To this end, we define the notion of \emph{deficiency}. In the following definition (and indeed the rest of the paper), for a graph $G$ and subgraph $H \subseteq G$, we will use the notation $v(G|H):= v(G)-v(H)$ and $e(G|H) := e(G)-e(H)$.

\begin{definition}
Let $G$ be a graph, and $H$ a subgraph of $G$. For a positive integer $g \geq 3$, we define the \emph{$g$-deficiency of $G$ with respect to $H$} as $\defc_g(G|H) := (g-2)e(G|H)-g\cdot v(G|H)$.
\end{definition}

The following quantity is also useful.
\begin{definition}
Let $G$ be a graph, and $H$ a subgraph of $G$. For a positive integer $g \geq 3$ and positive $\varepsilon \in \mathbb{R}$, we define $d_{g,\varepsilon}(G|H) : = (g-2)e(G|H) - (g+\varepsilon)v(G|H)$. Equivalently, $d_{g, \varepsilon} (G|H) := \defc_g(G|H)-\varepsilon \cdot v(G|H)$.
\end{definition}

(Note this nearly matches the definition of $d(\cdot)$ given in \cite{lukeevehyperbolicity}, ignoring the $b(\cdot)$ and $q(\cdot)$ terms.) 

We will require the following results from our previous paper \cite{lukeevehyperbolicity}. Theorem \ref{theorem:stronglinear} will be used in the proof of Theorem \ref{expmanycc}, and Observation \ref{girth5:stronglinear}, in the proof of Theorem \ref{expmanyccgirthfive}.  Theorem \ref{theorem:stronglinear} is a weaker version of one of the main theorems in \cite{lukeevehyperbolicity} (Theorem 3.21), ignoring the boundary and quasiboundary terms (and using $\alpha = \frac{1}{25}$ and $\gamma = \frac{7}{10}$).

The deficiency version of our hyperbolicity theorem is as follows.
\begin{thm}\label{theorem:stronglinear}

Let $\varepsilon = \frac{1}{50}$, let $G$ be a 2-connected plane graph, let $C$ be the outer cycle of $G$, and let $(L,M)$ be a 5-correspondence assignment for $G$. If $G$ is $C$-critical with respect to $(L,M)$ and $v(G|C)\ge 2$, then $d_{3, \varepsilon}(G|C)\geq \frac{23}{10}$.
\end{thm}

The following easy observation follows directly from our definitions.
\begin{obs}\label{deficiencyobs}
Let $G$ be a graph; let $H$ be a subgraph of $G$; and let $\varepsilon > 0$. If $d_{3, \varepsilon}(G|H) \geq 0$, then $v(G|H) \leq \varepsilon^{-1}\cdot \defc_3(G|H)$.
\end{obs}
\begin{proof}
By definition, $d_{3, \varepsilon}(G|H) = \defc_3(G|H)-\varepsilon \cdot v(G|H)$. Since $d_{3,\varepsilon }(G|H) \geq 0$, it follows that $0 \leq \defc_3(G|H)-\varepsilon \cdot v(G|H)$. Note that $\varepsilon > 0$; by isolating $v(G|H)$, we obtain the desired result.
\end{proof}
Note that the fact that $d_{3, \varepsilon}(G|C) \geq 0$ in Theorem \ref{theorem:stronglinear} implies Theorem \ref{theorem:mainhypthm}: using Euler's formula for graphs embedded in the plane and the definition of deficiency, we have that $\defc_3(G|C) \leq v(C)-3$. Combining this and Observation \ref{deficiencyobs}, one obtains Theorem \ref{theorem:mainhypthm}.

\subsection{Correspondence-Deletable Subgraphs} \label{subsec:corr-delsub}

This subsection introduces \emph{correspondence-deletable} graphs, defined below. A related notion (\emph{deletable graphs}) was defined for list colouring in \cite{postlelocalalgs}. We then show how to use our deficiency hyperbolicity theorem (Theorem \ref{theorem:stronglinear}) to obtain useful results concerning correspondence-deletable subgraphs. 
\begin{definition}\label{def:deletable}
Let $G$ be a graph, and let $H$ be an induced nonempty subgraph of $G$. We say $H$ is \emph{$r$-correspondence-deletable} if for every correspondence assignment $(L,M)$ of $H$ such that $|L(v)| \geq r-(\deg_G(v) - \deg_H(v))$ for each $v \in V(G)$, the graph $H$ has an $(L,M)$-colouring no matter the correspondence assignment $(L,M)$. 
\end{definition}

Note that if a graph is $r$-correspondence-deletable, it follows that it is also $r$-deletable (the list colouring analogue, defined in \cite{postlelocalalgs}). The converse, however, does not hold. 

If $H$ is an $r$-correspondence-deletable subgraph of $G$, then every $(L,M)$-colouring of $G\setminus V(H)$ extends to an $(L,M)$-colouring of $H$. In fact, the definition above captures an even stronger notion: for an $r$-correspondence-deletable subgraph $H \subseteq G$, an $(L,M)$-colouring of $G\setminus V(H)$ extends to an $(L,M)$-colouring of $H$ no matter the correspondence assignment $(L,M)$ and no matter the structure of $G \setminus V(H)$.

We highlight the relationship between critical graphs and deletable graphs: let $G$ be a graph, and $S$ a subgraph of $G$. If there exists a 5-correspondence assignment $(L,M)$ such that $G$ is $S$-critical with respect to $(L,M)$, then $G$ has no 5-correspondence-deletable subgraph that is vertex-disjoint from $S$. The converse does not hold. However, we can still use our deficiency hyperbolicity theorem to show that $G$ has high deficiency in this case. This result (below) will be instrumental in the proof of Theorem \ref{expmanycc}. Theorem \ref{d_gepsilonbound} will be proved later in this subsection.

\begin{restatable}{thm}{d_gepsilonbound}\label{d_gepsilonbound}
 Let $G$ be a plane graph, and let $\varepsilon$ be as in Lemma \ref{Hcritical}. If $H$ is a connected subgraph of $G$ such that there does not exist $X \subseteq V(G) \setminus V(H)$ such that $G[X]$ is 5-correspondence-deletable in $G$, then $d_{3,\varepsilon}(G|H) \geq 0$. 
\end{restatable}

We remark a list colouring version of Theorem \ref{d_gepsilonbound} was proved in \cite{postlelocalalgs} (Lemma 5.22). To prove Theorem \ref{d_gepsilonbound}, we will need the following two results. Again, list colouring versions of these results are proved in \cite{postlelocalalgs} (Theorem 5.20 and Proposition 5.21). The proof of Lemma \ref{Hcritical} is nearly identical to its list colouring analogue except for the use of Theorem \ref{theorem:stronglinear} (instead of a list colouring version of this theorem).

\begin{lemma}\label{Hcritical}
There exists $\varepsilon > 0$ such that following holds: if $G$ is a plane graph, $H$ is a connected subgraph of $G$, and there exists a 5-correspondence assignment $(L,M)$ for $G$ such that $G$ is $H$-critical with respect to $(L,M)$, then $d_{3,\varepsilon}(G|H) \geq 0$.
\end{lemma}

\begin{proof}
In the case where $v(G|H) = 0$, this follows from the definition of $H$-critical (since if $G$ is $H$-critical with $v(G|H) = 0$ then it follows that $e(G|H) \geq 1$). In the case where $v(G|H) =1$, again this follows from the definition of $H$-critical (as the vertex in $V(G) \setminus V(H)$ has at least five neighbours in $H$). The $v(G|H)\geq 2$ case follows from Theorem \ref{theorem:stronglinear}. Note that if $G$ is $H$-critical,  we may assume without loss of generality that $G$ is 2-connected: it suffices to delete edges from $H$ until it is a tree; to split each $v \in V(H)$ into $\deg(v)$ copies of itself, where each copy $v'$ is adjacent to copies $u'$, $w'$ of the two vertices $u,w$ that precede and follow it in the boundary walk of the tree obtained from $H$ (with multiplicity), as well as each vertex of $V(G)$ that lies between $u$ and $w$ in the portion of the cyclic ordering of the neighbours of $v$ (in the tree obtained from $H$) containing no other vertices of $H$. After performing this operation, $H$ becomes a facial cycle, and the result follows just as in Lemma 2.5 in \cite{luke5LC}.
\end{proof}

The proof of the following proposition is nearly identical to that of Proposition 5.21 in \cite{postlelocalalgs}; we include the proof for the purposes of cohesion.
\begin{prop}\label{notdelthencrit}
Let $G$ be a graph and $H$ a proper subgraph of $G$ such that $V(H) \neq V(G)$. If $G-V(H)$ is not $r$-correspondence-deletable in $G$, then there exists a subgraph $G_0$ of $G$ containing $H$ and an $r$-correspondence assignment for $G$ such that $G_0$ is $H$-critical.
\end{prop}
\begin{proof}[Proof (Adapted from Proposition 5.21, \cite{postlelocalalgs})]
Since $G-V(H)$ is not $r$-correspondence-deletable in $G$, there exists a correspondence assignment $(L_0, M_0)$ such that $|L_0(v)| \geq r-\deg_H(v)$ for each $v \in V(G) \setminus V(H)$ and $G-V(H)$ is not $(L_0, M_0)$-colourable. Define a new correspondence assignment $(L,M)$ of $G$ as follows. For each $v \in V(H)$, define $c_v$ to be a new colour not appearing in any other list. Define $R$ to be a set of $r-1$ distinct colours not appearing in any other list or in $\cup_{v \in V(H)} \{c_v\}$. Set $L(v) = \{c_v\} \cup R$ for each $v \in V(H)$. For each $u \in V(G) \setminus V(H)$, let $L(u) = L_0(u) \cup \{c_v: v \in N(u) \cap V(H)\}$. For each $uv \not \in E(H)$, set $M_{uv} = (M_0)_{uv}$. For each $uv \in E(H)$, set $M_{uv} = \emptyset$. Finally, for each $uv$ with $u \in V(G)\setminus V(H)$ and $v \in V(H)$, set $M_{uv} = \{(u,c_v)(v,c_v)\}$. Now $(L,M)$ is an $r$-correspondence assignment of $G$. Let $\phi$ be the colouring of $H$ given by $\phi(v) = c_v$ for every $v \in V(H)$. Since $G-V(H)$ is not $(L_0,M_0)$-colourable, it follows that $\phi$ does not extend to an $(L,M)$-colouring of $G$. Let $G'$ be an inclusion-wise minimal subgraph of $G$ containing $H$ such that $\phi$ does not extend to an $(L,M)$-colouring of $G'$. By the minimality of $G'$, we have that $\phi$ extends to an $(L,M)$-colouring of every proper subgraph of $G'$ containing $H$. Thus $G'$ is $H$-critical with respect to $(L,M)$, as desired.
\end{proof}

We are now equipped to prove Theorem \ref{d_gepsilonbound}.

\begin{proof}[Proof of Theorem \ref{d_gepsilonbound}]
We proceed by induction on $v(G|H)+e(G|H)$. If $V(H) = V(G)$, then $d_{3,\varepsilon}(G|H) \geq 0$ as desired. So we may assume that $V(H)\neq V(G)$. By assumption, $G-V(H)$ is not 5-correspondence-deletable in $G$. By Proposition \ref{notdelthencrit}, it follows that there exists a subgraph $G_0$ of $G$ containing $H$ and 5-correspondence assignment for $G_0$ such that $G_0$ is $H$-critical. Note that $H$ is a proper subgraph of $G_0$ by definition of $H$-critical. By Lemma \ref{Hcritical}, we have that $d_{3,\varepsilon}(G_0|H) \geq 0$. By definition of critical, since $H$ is connected it follows that $G_0$ is connected. Note that $v(G|G_0)+ e(G|G_0) < v(G|H) + e(G|H)$. Hence by induction, $d_{3,\varepsilon}(G|G_0) \geq 0$.  By definition of $d_{3, \varepsilon}$, we have that $d_{3,\varepsilon}(G|H) = d_{3,\varepsilon}(G|G_0) + d_{3,\varepsilon}(G_0|H) \geq 0 + 0 = 0$, as desired.
\end{proof}


\section{Counting 5-Correspondence Colourings}\label{sec:girth3}
In this section, we prove Theorem \ref{expmanyextensions} which we will show afterwards implies Theorem \ref{expmanycc}. Theorem \ref{expmanyextensions} involves counting the number of colouring extensions of a precoloured subgraph $S$ of a graph $G$ to a colouring of $G$ itself. The precise bound on the number of extensions is given in part in terms of the deficiency of $G$ with respect to $S$. The reader may find it helpful to consult Figure \ref{fig:expmany} while reading for a depiction of the cases considered in the proof.

The final tool we will need before proving Theorem \ref{expmanyextensions} is the following theorem, due to Thomassen. This theorem was originally written in the language of list colouring; however, as pointed out by Dvo{\v{r}}{\'a}k and the first author in \cite{dvovrak2018correspondence}, the proof also carries over to the realm of correspondence colouring.
\begin{thm}[Thomassen, \cite{thomassen5LC}]\label{thomtech5cc}
Let $G$ be a plane graph. Let $C$ be the subgraph of $G$ whose edge- and vertex-set are precisely those of the outer face boundary walk of $G$. Let $(L,M)$ be a correspondence assignment for $G$ where $|L(v)| \geq 1$ for a path $S \subseteq C$ with $v(S)\leq 2$; where $|L(v)| \geq 3$ for all $v \in V(C) \setminus V(S)$; and where $|L(v)|\geq 5$ for all $v \in V(G) \setminus V(C)$. Then every $(L,M)$-colouring of $S$ extends to an $(L,M)$-colouring of $G$.
\end{thm}

Here is our precolouring extension theorem (that implies Theorem \ref{expmanycc}). 
\begin{thm}\label{expmanyextensions}
Let $\varepsilon$ be as in Theorem \ref{d_gepsilonbound}. Let $G$ be a plane graph, let $S$ be a connected subgraph of $G$, and let $(L,M)$ be a 5-correspondence assignment for $G$. If $\phi$ is an $(L,M)$-colouring of $S$ that extends to an $(L,M)$-colouring of $G$,
then 
$$\log_2 E(\phi) \geq \frac{v(G|S) - (\varepsilon^{-1}+1)\defc_3(G|S)}{67},$$ 

where $E(\phi)$ denotes the number of extensions of $\phi$ to $G$.
\end{thm}
\begin{proof}

We proceed by induction on $v(G|S)$. First suppose $v(G|S) = 0$. Then since $\defc_3(G|S)$ is positive, the right-hand side of the equation is negative. As $\phi$ extends to an $(L,M)$-colouring of $G$ by assumption, we have that $\log_2 E(\phi)\ge 0$ and the result follows. Next suppose that $v(G|S) = 1$, and let $v \in V(G)\setminus V(S)$. Then 
$$\frac{v(G|S)-(\varepsilon^{-1}+1)\defc_3(G|S)}{67} = \frac{1-(\varepsilon^{-1}+1)(\deg(v)-3)}{67}.$$

Note that when $\deg(v) \geq 4$, the right-hand side is negative since $\varepsilon^{-1} > 0$. Since $\phi$ extends to an $(L,M)$-colouring of $G$ by assumption, it follows that $\log_2 E(\phi) \geq 0$, and so $\log_2 E(\phi) \geq \frac{v(G|S) - (\varepsilon^{-1}+1)\defc_3(G|S)}{67}$ holds as desired. We may therefore assume that $\deg(v) \leq 3$. Since $|L(v)| \geq 5$ it follows that $E(\phi) \geq 5-\deg(v)$.  Therefore it suffices to show that $\log_2 (5-\deg(v)) \geq \frac{1-(\varepsilon^{-1}+1)(\deg(v)-3)}{67}$, or equivalently, that $67 \geq \frac{1-(\varepsilon^{-1}+1)(\deg(v)-3)}{\log_2 (5-\deg(v))}$. The right-hand side is maximized when $\deg(v) = 0$, in which case it is easy to verify that $67 > \frac{154}{\log_2(5)}$. 

We may therefore assume that $v(G|S) \geq 2$.

Before proceeding with the remainder of the case analysis, we will need the following claim.
\begin{claim}\label{claimrdeletable}
 There does not exist a graph $H \subsetneq G$ with $S \subsetneq H$ and $v(S) < v(H) < v(G)$ such that $G-V(H)$ is a 5-correspondence-deletable subgraph of $G$. 
\end{claim}
\begin{proof}
Suppose not. Let $G' := G - V(H)$. Note that $G'$ is an induced subgraph of $G$. Since $v(S) < v(H)< v(G)$, we have that $v(H|S) < v(G|S)$, and so it follows by induction that there are at least $2^\frac{v(H|S)-(\varepsilon^{-1}+1)\defc_3(H|S)}{67}$ extensions of $\phi$ to $H$. Since $G'$ is 5-correspondence-deletable, we have by the definition of 5-correspondence-deletable subgraph that each of these extensions of $\phi$ to an $(L,M)$-colouring of $H$ extends further to an $(L,M)$ colouring of $G'$, and thus to $G$. Since $v(S) < v(H)< v(G)$, it follows that $v(G|H) < v(G|S)$, and so by induction for each extension of $\phi$ to an $(L,M)$ colouring $\phi'$ of $H$ there are at least $2^\frac{v(G|H)-(\varepsilon^{-1}+1)\defc_3(G|H)}{67}$ extensions of $\phi'$ to $G$. Therefore

\begin{align*}
 \log_2E(\phi) &\geq \frac{v(H|S)-(\varepsilon^{-1}+1)\defc_3(H|S)}{67} + \frac{v(G|H)-(
 \varepsilon^{-1}+1)\defc(G|H)}{67} \\
&=\frac{v(H|S)+v(G|H)-(\varepsilon^{-1}+1)(\defc_3(H|S)+\defc_3(G|H))}{67} \\
&= \frac{v(G|S)-(\varepsilon^{-1}+1)\defc_3(G|S)}{67},
\end{align*}
as desired.
\end{proof}

Among all vertices in $V(G)\setminus V(S)$, choose a vertex $v$ that maximizes $|N(v) \cap V(S)|$. Let $H := S + v$ (see Figure \ref{fig:expmany}). Since $\phi$ extends to an $(L,M)$-colouring of $G$, there is at least $1 = 2^0$ extension $\phi'$ of $\phi$ to $H$ where $\phi'$ extends further to an $(L,M)$-colouring of $G$. Since $\phi'$ extends to $G$ and $v(G|H) < v(G|S)$, we have by induction that there exist at least $2^\frac{v(G|H)-(\varepsilon^{-1}+1)\defc_3(G|H)}{67}$ extensions of $\phi'$ to $G$. Therefore 
$$
\log_2 E(\phi) \geq 0 + \frac{v(G|H)-(\varepsilon^{-1}+1)\defc_3(G|H)}{67}.$$

Moreover, since $v(G|H) = v(G|S)-1$ and $\defc_3(G|H) = \defc_3(G|S)-\deg_H(v)+3$, it follows that
$$\log_2 E(\phi) \geq \frac{v(G|S) - 1-(\varepsilon^{-1}+1)(\defc_3(G|S)-\deg_H(v)+3)}{67}. 
$$

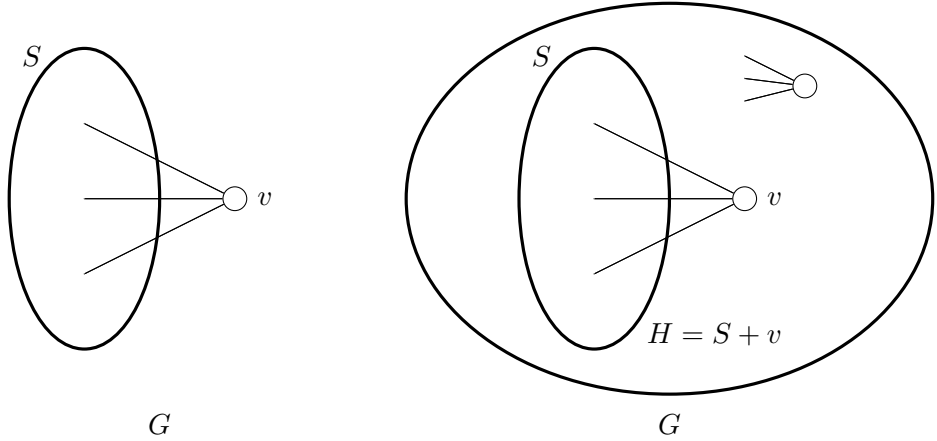
\begin{figure}[ht]
\tikzset{black/.style={shape=circle,draw=black,fill=black,inner sep=1pt, minimum size=9pt}}
\tikzset{white/.style={shape=circle,draw=black,fill=white,inner sep=1pt, minimum size=9pt}}

\tikzset{invisible/.style={shape=circle,draw=black,fill=black,inner sep=0pt, minimum size=0.1pt}}
\begin{center}
\begin{tikzpicture}
\filldraw[color=black!100, fill=black!0, very thick] (0,0) ellipse (1 and 2);

        \node[invisible] (1) at (0,-1) {};
        \node[invisible] (2) at (0,0) {};
        \node[invisible] (3) at (0,1) {};
        \node[white] (4) at (2,0) {};
        \node[]  at (2.4,0) {$v$};
        \node[] at (1,-3) {$G$};
        \node[] at (-0.7,1.9) {$S$};

        \draw[black] (4)--(1); 
        \draw[black] (4)--(2); 
        \draw[black] (4)--(3); 

\end{tikzpicture}
\hskip 15mm
\begin{tikzpicture}
\draw[color=black!100, very thick] (0,0) ellipse (1 and 2);
\draw[color=black!100, very thick] (1,0) ellipse (3.5 and 2.6);

        \node[invisible] (1) at (0,-1) {};
        \node[invisible] (2) at (0,0) {};
        \node[invisible] (3) at (0,1) {};
        \node[white] (4) at (2,0) {};
        \node[]  at (2.4,0) {$v$};
        \node[] at (1,-3) {$G$};
        \node[] at (1.6,-1.8) {$H = S + v$};
        \node[] at (-0.7,1.9) {$S$};
        
        \node[invisible] (5) at (2,1.9) {};
        \node[invisible] (6) at (2,1.6) {};
        \node[invisible] (7) at (2,1.3) {};
        \node[white] (8) at (2.8,1.5) {};

        \draw[black] (4)--(1); 
        \draw[black] (4)--(2); 
        \draw[black] (4)--(3); 
        \draw[black] (8)--(5); 
        \draw[black] (8)--(6); 
        \draw[black] (8)--(7); 
        
\end{tikzpicture}
\caption{The cases to consider for Theorem \ref{expmanyextensions}. First, the case where $V(G) = V(S) \cup \{v\}$; here, we consider each possible value of $\deg(v)$. For the case where $v(G) \geq v(S)+2$, we let $H = S+v$. Note then that $v(G)-v(S) > v(G)-v(H)$ and  $v(G)-v(S) > v(H)-v(S)$.}
    \label{fig:expmany}
\end{center}
\end{figure}

If $\deg_H(v) \geq 4$, the desired result immediately holds since $\varepsilon > 0$. Thus we may assume $\deg_H(v) \leq 3$. First suppose $\deg_H(v) \leq 2$.  We will show $G':= G-V(H)$ is 5-correspondence-deletable, contradicting Claim \ref{claimrdeletable}. To that end, let $(L',M')$ be a correspondence assignment for $G'$ with $|L'(u)| \geq 5-(\deg_G(u) - \deg_H(u))$ for all $u \in V(G')$. Note that by our choice of $v$, every vertex in $G'$ has $|L'(u)| \geq 2$. We claim $G'$ is $(L',M')$-colourable. To see this, let $T \subseteq V(G')$ be the set of vertices $\{u \in V(G'): |L'(u)| = 2\}$. Note that each vertex in $T$ is adjacent to $v$ in $G$. Let $G'':= G[V(G')\cup \{v\}]$. Let $(L'',M'')$ be a correspondence assignment for $G''$ obtained from $(L',M')$ as follows: define $L''$ by setting $L''(u) = L'(u)$ for all $u \in V(G')\setminus T$; setting $L''(v) = \{c\}$, where $c$ is a new colour not present in $\cup_{u \in V(G')}L(u)$; setting $L''(u) = L'(u) \cup \{c\}$ for all $u \in T$. Finally, define $M''$ by setting $M_{u_1u_2}'' = M'_{u_1u_2}$ for all $u_1u_2 \in E(G')$; and setting $M_{vu}'' = \{(v,c)(u,c)\}$ for all $u \in T$. Note that with the sole exception of $v$, every vertex $u$ in the outer face boundary walk of $G''$ has $|L''(u)| \geq 3$ by definition, and every vertex in $V(G'')\setminus (\{v\} \cup T)$ has $|L''(u)| \geq 5$. By Theorem \ref{thomtech5cc}, $G''$ has an $(L'',M'')$-colouring $\phi'$. Since no vertex $u \in T$ has $\phi'(u) = c$ by construction, it follows that $\phi'$ is an $(L',M')$-colouring of $G'$. This proves that $G' = G-V(H)$ is 5-correspondence-deletable; and since $v(S) < v(H) < v(G)$, this contradicts Claim \ref{claimrdeletable}.

We may therefore assume that $\deg_H(v) = 3$. 

By Claim \ref{claimrdeletable}, there does not exist $X \subseteq V(G)\setminus V(H)$ such that $G[X]$ is 5-correspondence-deletable in $G$. Thus by Theorem \ref{d_gepsilonbound}, we have that $d_{3,\varepsilon}(G|H) \geq 0$. Using this and Observation \ref{deficiencyobs}, it follows that 
\begin{equation}\label{vghbound}
    v(G|H) \leq \varepsilon^{-1}\defc_3(G|H).
\end{equation}
Moreover, since $v(G|S) = 1+v(G|H)$ and $\defc_3(G|S) = \defc_3(G|H) + \defc_3(H|S)$, it follows that
$$
   \frac{v(G|S)-(\varepsilon^{-1}+1)\defc_3(G|S)}{67} = \frac{1+v(G|H)-(\varepsilon^{-1}+1)(\defc_3(G|H)+\defc_3(H|S))}{67}.
$$ 

By definition of $\defc_3(G|S)$, we have that $\defc_3(H|S)= \deg_H(v)-3 = 0$ since $\deg_H(v)=3$. Using this and Observation \ref{deficiencyobs}, we obtain that 
 \begin{align*}
 \frac{v(G|S)-(\varepsilon^{-1}+1)\defc_3(G|S)}{67} & \leq \frac{1+\varepsilon^{-1}\defc_3(G|H)-(\varepsilon^{-1}+1)\defc_3(G|H)}{67}  \\
  &= \frac{1-\defc_3(G|H)}{67}. 
\end{align*}

By Theorem \ref{d_gepsilonbound}, we have that $\defc_3(G|H) \geq \varepsilon \cdot v(G|H)$, and since $v(G|S) \geq 2$, it follows that $v(G|H) \geq 1$ and so $\defc_3(G|H) \geq \varepsilon$. Since $\defc_3(G|H)$ is integral, it follows that $\defc_3(G|H) \geq 1$, and thus the right-hand side above is at most 0. Altogether then, we find that 

$$
\frac{v(G|S)-(\varepsilon^{-1}+1)\defc_3(G|S)}{67} \leq 0.
$$
But since $\phi$ extends to an $(L,M)$-colouring of $G$, we have that $\log_2(E(\phi)) \geq 0$. Thus
$$\log_2 E(\phi) \geq \frac{v(G|S)-(\varepsilon^{-1}+1)\defc_3(G|S)}{67},$$
as desired.
\end{proof}

We end this section by showing how Theorem \ref{expmanyextensions} implies Theorem \ref{expmanycc}. 

\begin{proof}[Proof of Theorem \ref{expmanycc}]
Let $S$ be the empty graph, and $\phi$ a trivial colouring of $S$. Let $E(\phi)$ denote the number of extensions of $\phi$ to $G$. 
Since $G$ is planar, we have that $e(G) \leq 3\cdot v(G)$. By Theorem \ref{expmanyextensions}, we find that
\begin{align*}
    \log_2{E(\phi)} &\geq \frac{v(G)-(50+1)(e(G)-3\cdot v(G))}{67} \\
    &\geq\frac{v(G)}{67},
\end{align*}
as desired.
\end{proof}

\section{Counting 3-Correspondence Colourings}\label{sec:girth5}

This section contains a proof of Theorem \ref{expmanyccgirthfive} which states that planar graphs of girth at least five have exponentially many 3-correspondence colourings. As in the girth three case, we prove Theorem \ref{expmanyccgirthfive} via a more technical theorem involving counting the number of colouring extensions of a precoloured subgraph. These proofs are found in Subsection \ref{subsec:proofsgirth5}. The tools needed to prove Theorem \ref{expmanyccgirthfive} are found in Subsection \ref{subsec:toolsgirth5} and are analogous to those needed to prove Theorem \ref{expmanycc}.

\subsection{Deficiency and Deletability Theorems for the Girth At Least Five Case}\label{subsec:toolsgirth5}

First, we need a deficiency version of a hyperbolicity theorem, given below. Note this is a slightly weaker version of a similar result from a previous paper of ours: namely, the result below corresponds to Observation 2 in \cite{lukeevehyperbolicity}, ignoring the $q(\cdot)$ term.

\begin{obs}\label{girth5:stronglinear}
Let $\varepsilon = \frac{1}{88}$; let $G$ be a planar graph of girth at least five; let $S$ be a connected subgraph of $G$; let $(L,M)$ be a 3-correspondence assignment for $G$. Suppose that $G$ is $S$-critical with respect to $(L,M)$, and that
\begin{itemize}
    \item $G$ is not composed of exactly $S$ and one edge not in $S$, and
    \item $G$ is not composed of exactly $S$ together with one vertex of degree $3$. 
\end{itemize}  
Then $d_{5, \varepsilon}(G|S) \geq 3$. 
\end{obs}

Below is a lemma analogous to Lemma \ref{Hcritical}.

\begin{lemma}\label{Hcritical5}
Let $\varepsilon = \frac{1}{88}$, let $G$ be a plane graph with girth at least five, and let $H$ be a connected subgraph of $G$. If there exists a 3-correspondence assignment $(L,M)$ for $G$ such that $G$ is $H$-critical with respect to $(L,M)$, then $d_{5, \varepsilon}(G|H) \geq 3$.
\end{lemma}
\begin{proof}
For $v(G|H) \geq 2$ or $v(G|H) = 1$ and $e(G|H) \neq 3$, this is directly implied by Observation \ref{girth5:stronglinear}.  Suppose now that $v(G|H) = 0$. Since $G$ is $H$-critical, it follows that $H$ is a proper subgraph of $G$ and so that $e(G|H) \geq 1$. Thus $d_{3,\varepsilon} \geq 3-0 = 3$. If $v(G|H) = 1$ and $e(G|H) = 3$, then $d_{5, \varepsilon}(G|H) = 3\cdot 3-5\cdot 1- \varepsilon \cdot 1 = 4-\varepsilon$. This is at least 3, since $\varepsilon = \frac{1}{88}$.
\end{proof}

Using this, we now establish a girth at least five version of Theorem \ref{d_gepsilonbound}. The proof is nearly identical to that of Theorem \ref{d_gepsilonbound}, and thus we omit it.

\begin{lemma}\label{d_gepsilonboundg5}
Let $\varepsilon$ be as in Lemma \ref{Hcritical5}. If $G$ is a plane graph with girth at least five and $H$ is a connected subgraph of $G$ such that there does not exist $X \subseteq V(G) \setminus V(H)$ such that $G[X]$ is 3-correspondence-deletable in $G$, then $d_{5, \varepsilon}(G|H) \geq 3$.
\end{lemma}

Similar to the case of Observation \ref{deficiencyobs}, this trivially implies the following.
\begin{obs}\label{deficiencyobs5}
Let $G$ be a graph of girth at least five, and $H$ a subgraph of $G$. If $d_{5,\varepsilon}(G|H) \geq 0$, then $v(G|H) \leq \varepsilon^{-1}\cdot (\defc_5(G|H)-3)$.
\end{obs}

Finally, we will need the following theorem, due to Thomassen. As in the girth three case, this theorem was originally written in the language of list colouring; however, as pointed out by Dvo{\v{r}}{\'a}k and the first author \cite{dvovrak2018correspondence}, the proof also carries over to the realm of correspondence colouring.
\begin{thm}[Thomassen, \cite{thomassen3LCnew}]\label{thomtech3cc}
Let $G$ be a plane graph of girth at least five. Let $C$ be the subgraph of $G$ whose edge- and vertex-set are precisely those of the outer face boundary walk of $G$. Let $(L,M)$ be a correspondence assignment for $G$ where $|L(v)| \geq 1$ for each vertex $v$ in a path or cycle $S \subseteq C$ with $v(S)\leq 6$; where $|L(v)| = 2$ for each vertex $v$ in an independent set $A$ of vertices in $V(C) \setminus V(S)$; where $|L(v)|\geq 3$ for all $v \in V(G) \setminus (A \cup V(S))$; and where there is no edge between vertices in $A$ and vertices in $S$. Then every $(L,M)$-colouring of $S$ extends to an $(L,M)$-colouring of $G$.
\end{thm}

\subsection{Proof of Theorem \ref{expmanyccgirthfive}}\label{subsec:proofsgirth5}
We now prove the following theorem, which is the girth at least five analogue to Theorem \ref{expmanyextensions}. The reader may find it helpful to consult Figure \ref{fig:expmany5} while reading for a depiction of the cases considered in the proof.

\begin{thm}\label{expmanyextensions5}
Let $G$ be a plane graph of girth at least five, let $S$ be a connected subgraph of $G$, and let $(L,M)$ be a 3-correspondence assignment for $G$. If $\phi$ is an $(L,M)$-colouring of $S$ that extends to an $(L,M)$-colouring of $G$,
then 
$$\log_2 E(\phi) \ge \frac{v(G|S) - 89\cdot\defc_5(G|S)}{282},$$
where $E(\phi)$ denotes the number of extensions of $\phi$ to $G$.

\end{thm}
\begin{proof}

We proceed by induction on $v(G|S)$. First suppose that $v(G|S) = 0$. Then since $\defc_5(G|S)$ is positive, the right-hand side of the equation is negative. Since $\phi$ extends to an $(L,M)$-colouring of $G$ by assumption, we have that $\log_2 E(\phi) \geq 0$ and the result follows. Next suppose that $v(G|S) = 1$, and let $v \in V(G)\setminus V(S)$. Then 
\[ \frac{v(G|S)-89\cdot \defc_5(G|S)}{282}    = \frac{1-89(3\cdot \deg(v)-5)}{282}.
\]

Note that when $\deg(v) \geq 2$, the right-hand side is negative. Since $\phi$ extends to an $(L,M)$-colouring of $G$ by assumption, it follows that $\log_2 E(\phi) \geq 0$, and so 

$$\log_2 E(\phi) \ge \frac{v(G|S) - 89\cdot \defc_5(G|S)}{282},$$ as desired. Thus we may assume that $\deg(v) \leq 1$. Since $|L(v)| \geq 3$, we have that $E(\phi) \geq 3-\deg(v)$. Therefore it suffices to show that $\log_2 (3-\deg(v)) \geq \frac{1-89(3\cdot\deg(v)-5)}{282}$; or, equivalently, that
$$
    282 \geq \frac{1-89(3\cdot \deg(v)-5)}{\log_2(3-\deg(v))}.
$$
When $\deg(v) = 0$, the right-hand side equals $\frac{446}{\log_2(3)} < 282$. When $\deg(v) = 1$, the right-hand side equals $179$, which again is less than $282$. Thus the inequality above holds.

We may therefore assume that $v(G|S) \geq 2$.

Before proceeding with the remainder of the case analysis, we will need the following claim.
\begin{claim}\label{claimrdeletable5}
There does not exist a graph $H \subsetneq G$ with $S \subsetneq H$ and $v(S) < v(H) < v(G)$ such that $G-V(H)$ is a 3-correspondence-deletable subgraph of $G$. 
\end{claim}

\begin{proof}
Suppose not. Let $G' := G - V(H)$ be a 3-correspondence-deletable subgraph. Note that $G'$ is induced. Since $S \subsetneq H$ and $H \subsetneq G$ and $v(S) < v(H) < v(G)$, it follows that $v(H|S) < v(G|S)$. By induction, there are at least $2^\frac{v(H|S)-89\cdot \defc_5(H|S)}{282}$ extensions of $\phi$ to $H$. Since $G'$ is a 3-correspondence-deletable subgraph of $G$, we have by definition that each of these extensions of $\phi$ to an $(L,M)$-colouring of $H$ extends further to an $(L,M)$ colouring of $G'$, and therefore to $G$.

Since $H \subsetneq G$ and $S \subsetneq H$ and $v(S) < v(H) < v(G)$, we have that $v(G|H) < v(G|S)$, and so by induction for each extension of $\phi$ to an $(L,M)$ colouring $\phi'$ of $H$ there are at least $2^\frac{v(G|H)-89\cdot \defc_5(G|H)}{282}$ extensions of $\phi'$ to $G$. Therefore

\begin{align*}
 \log_2E(\phi) &\geq \frac{v(H|S)-89\cdot \defc_5(H|S)}{282} + \frac{v(G|H)-89\cdot \defc_5(G|H)}{282} \\
&=\frac{v(H|S)+v(G|H)-89(\defc_5(H|S)+\defc_5(G|H))}{282} \\
&= \frac{v(G|S)-89\cdot \defc_5(G|S)}{282},
\end{align*}
as desired.
\end{proof}

Among all vertices in $V(G)\setminus V(S)$, choose a vertex $v$ that maximizes $|N(v) \cap V(S)|$. Let $H := S + v$, and let $G' := G- V(H)$. Since $\phi$ extends to an $(L,M)$-colouring of $G$, there is at least $1 = 2^0$ extension $\phi'$ of $\phi$ to $H$ where $\phi'$ extends further to an $(L,M)$-colouring of $G$. Since $\phi'$ extends to $G$ and $v(G|H) < v(G|S)$, by induction we find that there exist at least $2^\frac{v(G|H)-89\cdot \defc_5(G|H)}{282}$ extensions of $\phi'$ to $G$. Therefore 
$$
\log_2 E(\phi) \geq 0 + \frac{v(G|H)-89\cdot \defc_5(G|H)}{282}.
$$

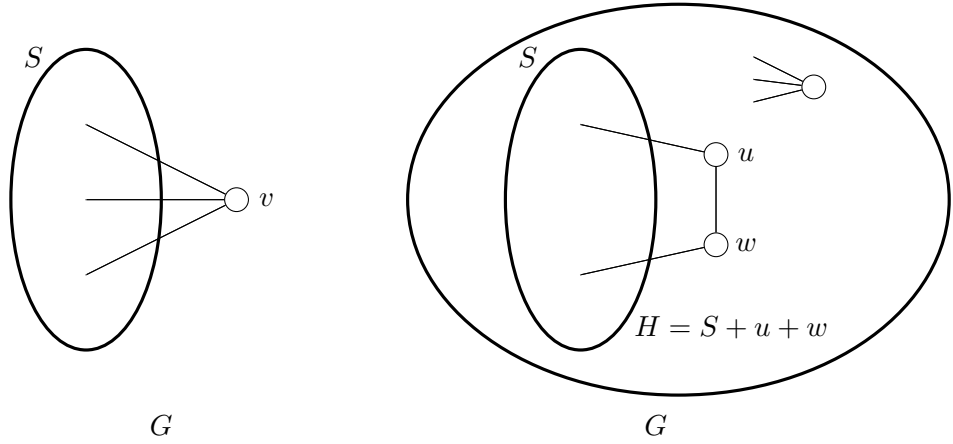
\begin{figure}[ht]
\tikzset{black/.style={shape=circle,draw=black,fill=black,inner sep=1pt, minimum size=9pt}}
\tikzset{white/.style={shape=circle,draw=black,fill=white,inner sep=1pt, minimum size=9pt}}

\tikzset{invisible/.style={shape=circle,draw=black,fill=black,inner sep=0pt, minimum size=0.1pt}}
\begin{center}
\begin{tikzpicture}
\filldraw[color=black!100, fill=black!0, very thick] (0,0) ellipse (1 and 2);

        \node[invisible] (1) at (0,-1) {};
        \node[invisible] (2) at (0,0) {};
        \node[invisible] (3) at (0,1) {};
        \node[white] (4) at (2,0) {};
        \node[]  at (2.4,0) {$v$};
        \node[] at (1,-3) {$G$};
        \node[] at (-0.7,1.9) {$S$};

        \draw[black] (4)--(1); 
        \draw[black] (4)--(2); 
        \draw[black] (4)--(3); 

\end{tikzpicture}
\hskip 15mm
\begin{tikzpicture}
\draw[color=black!100, very thick] (0,0) ellipse (1 and 2);
\draw[color=black!100, very thick] (1.3,0) ellipse (3.6 and 2.6);

        \node[invisible] (1) at (0,1) {};
        \node[white] (4) at (1.8,0.6) {};
        \node[white] (9) at (1.8,-0.60) {};
        \node[invisible] (12) at (0,-1) {};

        \node[]  at (2.2,0.6) {$u$};
        \node[]  at (2.2,-0.6) {$w$};

        \node[] at (1,-3) {$G$};
        \node[] at (2.0,-1.7) {$H = S + u+w$};
        \node[] at (-0.7,1.9) {$S$};
        
        \node[invisible] (5) at (2.3,1.9) {};
        \node[invisible] (6) at (2.3,1.6) {};
        \node[invisible] (7) at (2.3,1.3) {};
        \node[white] (8) at (3.1,1.5) {};

        \draw[black] (4)--(1); 
        \draw[black] (9)--(12);   
        \draw[black] (9)--(4);         
        \draw[black] (8)--(5); 
        \draw[black] (8)--(6); 
        \draw[black] (8)--(7); 
        
\end{tikzpicture}
\caption{Two of the cases to consider for Theorem \ref{expmanyextensions5}. First, the case where $V(G) = V(S) \cup \{v\}$; here, we consider each possible value of $\deg(v)$. The case where $v(G) = v(S)+ 2$ is easily dealt with, and is not pictured. For the case where $v(G) \geq v(S)+3$, we let $H = S+u+w$. Note then that $v(G)-v(S) > v(G)-v(H)$ and  $v(G)-v(S) > v(H)-v(S)$.}
    \label{fig:expmany5}
\end{center}
\end{figure}

Since $\defc_5(G|H) = \defc_5(G|S)-3\cdot \deg_H(v)+5$, it follows that
\begin{align*}
\log_2 E(\phi) &\geq \frac{v(G|S) - 1-89(\defc_5(G|S)-3\cdot \deg_H(v)+5)}{282} \\
&= \frac{v(G|S)-1-89\cdot \defc_5(G|S)}{282} - \frac{89(5-3\cdot \deg_H(v))}{282}.
\end{align*}
If $\deg_H(v) \geq 2$, then $ \log_2 E(\phi) \geq \frac{v(G|S)-1-89\cdot \defc_5(G|S)}{282}$, as desired. Thus we may assume $\deg_H(v) \leq 1$. Suppose $\deg_H(v) = 0$. We will show $G'$ is 3-correspondence-deletable, contradicting Claim \ref{claimrdeletable5}. To that end, let $(L',M')$ be a correspondence assignment for $G'$ with $|L'(u)| \geq 3-(\deg_G(u) - \deg_H(u))$ for all $u \in V(G')$. We claim $G'$ is $(L',M')$-colourable. To see this, first note that by our choice of $v$, every vertex in $G'$ has $|L'(u)| \geq 2$. Let $T \subseteq V(G')$ be the set of vertices $\{u \in V(G'): |L'(u)| = 2\}$. Since $\deg_H(v) = 0$, by our choice of $v$ each of the vertices $u \in T$ is adjacent to $v$ in $G$. Since $G$ has girth at least five, it follows that  $T$ is an independent set. By Theorem \ref{thomtech3cc}, $G'$ has an $(L',M')$-colouring $\phi'$. This proves that $G'$ is 3-correspondence-deletable; and since $v(S) < v(H) < v(G)$, this contradicts Claim \ref{claimrdeletable5}.

We may therefore assume that $\deg_H(v) = 1$. First suppose that $v(G|S) = 2$. Let $\{w\} = V(G) \setminus V(H)$. Note that in this case $G'=G- V(H)$ consists only of the vertex $u$. Since $v \in V(G) \setminus V(S)$ was chosen to maximize $|N(v) \cap V(S)|$, it follows that $|N(w) \cap V(S)| \leq 1$. Since $|L(u)| \geq 3$ and $\deg_G(w) \leq 2$, it follows that $G'$ is 3-correspondence-deletable, again contradicting Claim \ref{claimrdeletable5}.

Thus we assume that $v(G|S) \geq 3$. Let $X$ be the set of vertices in $V(G) \setminus V(S)$ with at least one neighbour in $S$. Note that $v \in X$.  By our choice of $v$, since $\deg_H(v) = 1$ it follows that every vertex in $X$ has exactly one neighbour in $S$. First suppose that $X$ is an independent set. We will show that $G'$ is a 3-correspondence-deletable subgraph of $G$, again contradicting Claim \ref{claimrdeletable5}. To that end, let $(L',M')$ be a correspondence assignment for $G'$ with $|L'(u)| \geq 3-(\deg_G(u) - \deg_H(u))$ for all $u \in V(G')$. Note that by our choice of $v$ and the fact that $X$ is an independent set, it follows that every vertex in $G'$ has $|L'(u)| \geq 2$. We claim $G'$ is $(L',M')$-colourable. To see this, let $G'' := G[V(G')\cup \{v\}]$. Note that in $G$, each vertex in $X$ and exactly one vertex in $S$. It follows that the vertices of $X$ are in the outer face boundary walk of $G''$. Let $(L'',M'')$ be a correspondence assignment for $G''$ obtained from $(L',M')$ as follows: define $L''$ by setting $L''(u) = L'(u)$ for all $u \in V(G')\setminus N_{G''}(v)$; setting $L''(v) = \{c\}$, where $c$ is a new colour not present in $\cup_{u \in V(G')}L(u)$; setting $L''(u) = L'(u) \cup \{c\}$ for all $u \in N_{G''}(v)$. Finally, define $M''$ by setting $M_{u_1u_2}'' = M'_{u_1u_2}$ for all $u_1u_2 \in E(G')$; and setting $M_{vu}'' = \{(v,c)(u,c)\}$ for all $u \in N_{G''}(v)$. Note that with the sole exception of $v$, every vertex $u$ in the outer face boundary walk of $G''$ has $|L''(u)| \geq 2$ by definition, and every vertex in $V(G'')\setminus (\{v\} \cup X)$ has $|L''(u)| \geq 3$. Moreover, by assumption, $X$ is an independent set. By Theorem \ref{thomtech3cc}, $G''$ has an $(L'',M'')$-colouring $\phi'$. Since no vertex $u \in T$ has $\phi'(u) = c$ by construction, it follows that $\phi'$ is an $(L',M')$-colouring of $G'$. This proves that $G' = G-V(H)$ is 3-correspondence-deletable; and since $v(S) < v(H) < v(G)$, this contradicts Claim \ref{claimrdeletable5}.

Thus $X$ is not an independent set, and so there exist vertices $u,w \in X$ such that $uw \in E(G)$. Let $H' := S + u + w$ (see Figure \ref{fig:expmany5}). Since every vertex in $X$ has exactly one neighbour in $S$, it follows that $e(H'|S) = 3$. Thus $\defc_5(G|H') = \defc_5(G|S)+1$; and by Claim \ref{claimrdeletable5}, there does not exist $X \subseteq V(G) \setminus V(H')$ such that $G[X]$ is 3-correspondence-deletable in $G$.  Thus by Lemma \ref{d_gepsilonboundg5}, we have that $d_{5,\varepsilon}(G|H') \geq 3$, and so by Observation \ref{deficiencyobs5}, $\defc_5(G|H') \geq \varepsilon \cdot v(G|H') + 3$. Thus $\defc_5(G|S) \geq 2 + \varepsilon \cdot v(G|H')$. It follows that

\begin{align*}
    v(G|S)-89\cdot \defc_5(G|S) & \leq v(G|S) -89(2 + \varepsilon \cdot  v(G|H)) \\
    &= v(G|S)-89(2 + \varepsilon \cdot (v(G|S)-2)) \\
    &= v(G|S)-89\varepsilon \cdot v(G|S)-89(2-2\varepsilon).
\end{align*}

As $\varepsilon = \frac{1}{88}$, the above is negative.  Thus 
$$
0 > \frac{v(G|S)-89\cdot \defc_5(G|S)}{282}.
$$

Since $\phi$ extends to an $(L,M)$-colouring of $G$, it follows that $\log_2 E(\phi) \geq 0$, and so $\log_2 E(\phi) > \frac{v(G|S)-89\cdot \defc_5(G|S)}{282}$, as desired. 
\end{proof}

As an easy corollary, we obtain Theorem \ref{expmanyccgirthfive}.

\begin{proof}[Proof of Theorem \ref{expmanyccgirthfive}]
Let $S$ be the empty graph, and $\phi$ a trivial colouring of $S$. Let $E(\phi)$ denote the number of extensions of $\phi$ to $G$. Since $G$ is planar and has girth at least five, $e(G)\leq \frac{5}{3} \cdot v(G)$. By Theorem \ref{expmanyextensions5},

\begin{align*}
    \log_2 E(\phi) & \geq \frac{v(G)-89(3\cdot e(G)-5\cdot v(G))}{282} \\
    &\geq \frac{v(G)}{282},
\end{align*}
as desired.
\end{proof}
\section{Locally Planar Graphs}\label{sec:locallyplanar}

The main theorems of this section are Theorems \ref{thm:expmanylocallyplanar} and \ref{thm:expmanylocallyplanarg5}, restated below, which show that for each $k \in \{3,5\}$, locally planar graphs of girth $(8-k)$ have exponentially many $k$-correspondence colourings.

\thmexpmanyloc*
\thmexpmanylocfive*

We prove these results by first proving technical theorems (Theorem \ref{thm:localplanarhypversion} for the 5-correspondence case and \ref{thm:localplanarhypversion-g5} for the 3-correspondence case) and combining these technical theorems with a theorem of Postle and Thomas \cite{postle2018hyperbolic}, stated below. Before stating these technical theorems, we require the following definition.

\begin{definition}
    Let $\varepsilon > 0$,  let $G$ be a graph, and let $H$ be an induced subgraph of $G$. We say $H$ is \emph{$\varepsilon$-exponentially-$r$-correspondence-deletable} if for every correspondence assignment $(L,M)$ of $H$ such that $|L(v)| \geq r-(\deg_G(v) - \deg_H(v))$ for each $v \in V(G)$, the graph $H$ has at least $2^{\varepsilon v(H)}$ distinct $(L,M)$-colourings no matter the correspondence assignment $(L,M)$. 
\end{definition}

Here are the technical theorems mentioned above.

\begin{thm}\label{thm:localplanarhypversion}
    Let $\mathcal{F}$ be a family of embedded graphs where for each $(G, \Sigma) \in \mathcal{F}$, the graph $G$ does not contain a subset $X \subseteq V(G)$ such that $G[X]$ is $\frac{1}{3484}$-exponentially-5-correspondence-deletable. Then $\mathcal{F}$ is hyperbolic, and $52$ is a Cheeger constant for $\mathcal{F}$.
\end{thm}
\begin{thm}\label{thm:localplanarhypversion-g5}
    Let $\mathcal{F}$ be a family of embedded graphs of girth at least five where for each $(G, \Sigma) \in \mathcal{F}$, the graph $G$ does not contain a subset $X \subseteq V(G)$ such that $G[X]$ is $\frac{1}{25380}$-exponentially-5-correspondence-deletable. Then $\mathcal{F}$ is hyperbolic, and $270$ is a Cheeger constant for $\mathcal{F}$.
\end{thm}

Finally, we will invoke the following theorem due to Postle and Thomas.

\begin{thm}[Postle and Thomas, \cite{postle2018hyperbolic}]\label{thm:smalledgewidth}
For every hyperbolic family $\mathcal{F}$ that is closed under curve cutting there exists a constant $k>0$ such that if $(G,\Sigma) \in \mathcal{F}$ and $\Sigma$ has Euler genus $g$, then $G$ contains a non-contractible cycle of length at most $k\log(g+1)$.    
\end{thm}

Theorem \ref{thm:expmanylocallyplanar} follows easily from Theorem \ref{thm:localplanarhypversion}; the proof of Theorem \ref{thm:expmanylocallyplanarg5} assuming Theorem \ref{thm:localplanarhypversion-g5} is identical, and as such we omit it.

\begin{proof}[Proof of Theorem \ref{thm:expmanylocallyplanar} (assuming Theorem \ref{thm:localplanarhypversion}).] 
Let $\mathcal{F}$ be the family of embedded graphs where for each $(H,\Sigma) \in \mathcal{F}$, the graph $H$ does not contain a subgraph that is $\frac{1}{3484}$-exponentially-5-correspondence-deletable.  By Theorem \ref{thm:localplanarhypversion}, $\mathcal{F}$ is a hyperbolic family. Fix a surface $\Sigma$ of Euler genus $g$.  By Theorem \ref{thm:smalledgewidth}, there exists a constant $k > 0$ such that every graph in $\mathcal{F}$ embedded in $\Sigma$ has edge-width at most $k \log(g+1)$. Let $(G, \Sigma)$ be an embedded graph with edge-width greater than $k \log (g+1)$, and let $(L,M)$ be a 5-correspondence assignment for $G$. We will prove by induction on $v(G)$ that $G$ has exponentially-many distinct $(L,M)$-colourings. 

First suppose $v(G) \in \{0,1\}$. Then $G$ trivially has at least $2^{\frac{v(G)}{3484}}$ distinct $(L,M)$-colourings, as desired. Next suppose $v(G) \geq 2$, that $G$ has edge-width greater than $k\log(g+1)$, and that the statement holds for all embedded graphs $(G', \Sigma)$ of edge-width greater than $k \log (g+1)$ with fewer vertices. Since every graph in $\mathcal{F}$ embedded in $\Sigma$ has edge-width at most $k\log(g+1)$, it follows that $G$ contains a subgraph $(H, \Sigma)$ that is $\frac{1}{3484}$-exponentially-5-correspondence-deletable. Note that $G-H$ is a graph embeddable in $\Sigma$ with edge-width greater than $k\log(g+1)$. By induction, $G-H$ has at least $2^\frac{v(G|H)}{3484}$ distinct $(L,M)$-colourings. By definition of $\frac{1}{3484}$-exponentially-5-correspondence-deletable, every $(L,M)$-colouring of $G-H$ extends to at least $2^\frac{v(H)}{3438}$ distinct $(L,M)$-colourings of $H$. Thus $G$ has at least $2^\frac{v(G|H)+v(H)}{3484}$ distinct $(L,M)$-colourings, as desired.
\end{proof}

Our final point of order is thus to prove Theorem \ref{thm:localplanarhypversion}. We will omit the proof of Theorem \ref{thm:localplanarhypversion-g5} as it is nearly identical; key differences will be discussed at the end of this section. Before proceeding, we will require the following easy results which follow from the definition of $r$-correspondence-deletability.

\begin{lemma}\label{lemma:maxdel}
    Let $G$ be a graph. If $H$ is an $r$-correspondence-deletable subgraph of $G$, and $G-V(H)$ contains an $r$-correspondence-deletable subgraph $H'$, then $G[V(H) \cup V(H')]$ is an $r$-correspondence-deletable subgraph of $G$. 
\end{lemma}

The following is an immediate corollary of Lemma \ref{lemma:maxdel}.
\begin{cor}\label{cor:maxdel}
    If $H$ is a maximal $r$-correspondence-deletable subgraph of $G$, then $G-V(H)$ does not contain an $r$-correspondence-deletable subgraph.
\end{cor}

We now prove Theorem \ref{thm:localplanarhypversion}.

\begin{proof}[Proof of Theorem \ref{thm:localplanarhypversion}]
   Let $(G, \Sigma) \in \mathcal{F}$, and let $\eta$ be a closed curve in $\Sigma$ that intersects $G$ only in vertices and that bounds an open disk $\Delta$. We may assume that $\eta$ is a simple curve; otherwise, we split vertices intersected by $\eta$ into multiple copies to reduce to this case. Let $D$ be the set of vertices of $G$ intersected by $\eta$. Suppose that $\Delta$ includes a vertex of $G$. We will show the number vertices of $G$ in $\Delta$ is at most $52(|D|-1)$.

    Let $G_\Delta$ be the subgraph of $G$ embedded in $\Delta$ and its boundary, $\eta$. For each pair of consecutive vertices $u,v$ along $\eta$, add the edge $uv$ if $uv \not \in E(G)$, so that $G_\Delta[D]$ is connected. Let $G_D := G_\Delta[D]$. Let $\varepsilon$ be as in Theorem \ref{d_gepsilonbound}, and let $\alpha := \frac{\varepsilon}{2\varepsilon + 1}$ (that is, $\varepsilon = \frac{1}{50}$ and hence $\alpha := \frac{1}{52}$).
    
    Let $G' \subseteq G_\Delta$ be a maximal connected subgraph with $D \subseteq V(G')$ and $d_{3,\alpha}(G'|G_D) \geq 0$. Note that such a subgraph $G'$ exists, since $d_{3,\alpha}(G_D|G_D) = 0$. Moreover, it follows that $G'$ is an induced subgraph of $G_\Delta$ since adding edges does not decrease $d_{3, \alpha}(G'|G_D)$.

    Note that by definition of $d_{3,\alpha}$, we have that since $G_D \subseteq G' \subseteq G_\Delta$, 
    $$
    d_{3,\alpha}(G_\Delta|G_D) = d_{3,\alpha}(G_\Delta|G')+ d_{3,\alpha}(G'|G_D).
    $$

    We aim to show that $d_{3,\alpha}(G_\Delta|G_D) \geq 0$, which we will show later implies that the number of vertices of $G$ in $\Delta$ is at most $52(|D|-1)$. To that end, we prove the following.
    
    \begin{claim}
         $d_{3,\alpha}(G_\Delta|G') \geq 0$.
    \end{claim}
    \begin{proof}
    To see this, suppose not. Then since $d_{3,\alpha}(G_\Delta|G') < 0$ and $\alpha < \varepsilon$, it follows that $d_{3,\varepsilon}(G_\Delta|G') < 0$. By the contrapositive of Theorem \ref{d_gepsilonbound}, there exists a nonempty subset $X \subseteq V(G_\Delta)-V(G')$ such that $G_\Delta[X]$ is 5-correspondence-deletable subgraph of $G_\Delta$. We choose $X \subseteq G_\Delta-V(G')$ to be a maximal such set, and let $H = G[X]$. 
    
    Note that since $H$ is maximal, by Corollary \ref{cor:maxdel} there does not exist a 5-correspondence-deletable subgraph of $G_\Delta-V(H)$ disjoint from $G'$. Thus by Theorem \ref{d_gepsilonbound}, we have that $d_{3, \varepsilon}(G_\Delta-V(H)~|~G') \geq 0$, and so $d_{3, \alpha}(G_\Delta-V(H)~|~G') \geq 0$ since $\alpha < \varepsilon$. Now, since $G_D \subseteq G' \subseteq G_\Delta-V(H)$, by definition of $d_{3,\alpha}$ we have that $d_{3, \alpha}(G_\Delta-V(H)~|~G_D) = d_{3, \alpha}(G_\Delta-V(H)~|~G') + d_{3, \alpha}(G'|G_D)$. As the right hand side is non-negative, so is the left hand side; and so by our choice of $H$, we have that $G_\Delta-V(H) =G'$.
    
    Since $(G, \Sigma) \in \mathcal{F}$, we have that (though $H$ is 5-correspondence-deletable) $H$  is not $\frac{1}{3484}$-exponentially-5-correspondence-deletable.
    Thus there exists a correspondence assignment $(L,M)$ of $H$ where $|L(u)| \geq 5-(\deg_{G_\Delta}(u) - \deg_H(u))$ for each $u \in V(H)$ and such that $H$ has an $(L,M)$-colouring, but does not have at least $2^{\frac{v(H)}{3484}}$ distinct $(L,M)$-colourings.
    
   Let $(L',M')$ be the 5-correspondence assignment for $G_{\Delta}$ defined as follows. For each vertex $v \in V(G')$, let $\varphi(v)$ be a new unique colour for $v$. For $v\in V(G')$, we let $L'(v) := \{\phi(v)\}$. For $u\in V(H)$, we let $L'(u) := L(u) \cup \{\varphi(v): v \in V(G') \cap N_{G_\Delta}(u)\}$. For $e=uv\in E(G_{\Delta})$, we let $M'_e = M_e$ if $e\in E(H)$, and $M'_e=\emptyset$ if $e\in E(G')$; finally for $e\in E(G_{\Delta})\setminus (E(H)\cup E(G'))$ where we have that $|\{u,v\}\cap V(H)|=1$, we\textemdash assuming without loss of generality that $u\in V(H)$ and $v\in V(G')$\textemdash let $M'_{uv} = \{(u, \varphi(v))(v, \varphi(v))\}$. Recall that $G'$ is connected by definition; and thus by Theorem \ref{expmanyextensions} applied to $G_\Delta$, $G'$, and $(L',M')$,
\begin{align*}
\log_2 E(\varphi) &\geq \frac{v(G_\Delta|G')-(\varepsilon^{-1}+1)\defc_3(G_\Delta|G
')}{67}\\
&=\frac{v(H)-(\varepsilon^{-1}+1)\defc_3(G_\Delta|G')}{67}
\end{align*}

Since $d_{3,\alpha}(G_{\Delta}|G') = \defc_3(G_{\Delta}|G') - \alpha \cdot v(G_\Delta|G') < 0$, it follows that $\defc_3(G_\Delta | G') < \alpha \cdot v(G_\Delta|G')$; equivalently, that $\defc_3(G_\Delta|G') < \alpha \cdot v(H)$.

Thus, continuing from above,
\begin{align*}
 \log_2 E(\varphi)   &> \frac{v(H)-(\varepsilon^{-1}+1)\alpha\cdot v(H)}{67} \\
 &= \frac{(1-\alpha(\varepsilon^{-1}+1))v(H)}{67} \\
 &= \frac{v(H)}{52 \cdot 67} \textrm{ \hskip 8mm since $\alpha = \frac{1}{52}$ and $\varepsilon = \frac{1}{50}$.}
\end{align*}
As every extension of $\varphi$ to $H$ corresponds to a distinct $(L,M)$-colouring of $H$, we have that $H$ has at least $2^{\frac{v(H)}{3484}}$ distinct $(L,M)$-colourings, a contradiction.
\end{proof}

Thus $d_{3,\alpha}(G_\Delta|G') \geq 0$, and so $d_{3,\alpha}(G_\Delta|G_D) \geq 0$. We now show this implies the number of vertices in $G_\Delta-D$ is at most $52(|D|-1)$. To see this, let $F(G_\Delta)$ be the set of faces in the embedding of $G_\Delta$, and let $f(G_\Delta) = |F(G_\Delta)|$. Recall that $G_D$ is connected. We claim moreover that $G_\Delta$ is connected; otherwise, by Theorem \ref{expmanycc} a component $C$ of $G$ has at least  $2^\frac{v(C)}{67}$ distinct $(L,M)$-colourings for every 5-correspondence assignment $(L,M)$. But $C$ is therefore $\frac{1}{3484}$-exponentially-5-correspondence-deletable, contradicting that $G \in \mathcal{F}$.  Moreover, we we may assume $G_\Delta$ has at least three vertices: this too follows easily from the fact that $G$ does not contain a $\frac{1}{3484}$-exponentially-5-correspondence-deletable subgraph. Since $G_\Delta$ is connected and $v(G_\Delta) \geq 3$, the boundary walk of every face has length at least three. Let $f_O$ be the outer face of $G_\Delta$. By Euler's formula for graphs embedded in surfaces,
\begin{align*}
    e(G_\Delta)-3v(G_\Delta) &= e(G_\Delta)-3(e(G_\Delta)+2-f(G_\Delta)) \\
    &= 3f(G_\Delta)-2e(G_\Delta)-6 \\
    &= \sum_{f \in F(G_\Delta)} (3-|f|)-6 \\
    &\leq (3-|D|)-6 \textrm{\hskip 4mm since $|f_O| \geq |D|$, and $3-|f| \leq 0$ for all other $f \in F(G_\Delta)$.} \\
    &= -|D|-3.
\end{align*}
Now, since $\defc_3(G_\Delta|G_D) = e(G_\Delta|G_D)-3v(G_\Delta|G_D)$, it follows from above that $\defc_3(G_\Delta|G_D) \leq -|D|-3 -e(G_D)+3|D| = 2|D|-e(G_D) -3$. By construction, $e(G_D) \geq |D|-1$, and so $\defc_3(G_\Delta|G_D) \leq |D|-2$. Since $d_{3,\alpha}(G_\Delta|G_D) = \defc_3(G_\Delta|G_D)-\alpha \cdot v(G_\Delta|G_D)$ and $d_{3, \alpha}(G_\Delta|G_D) \geq 0$, it follows that $(|D|-2)-\alpha \cdot v(G_\Delta|G_D) \geq 0$. Rearranging, we have that  $v(G_\Delta|G_D) \leq \alpha^{-1}(|D|-2)$. Thus $\mathcal{F}$ is a hyperbolic family, and 
$\frac{1}{\alpha} = 52$ is a Cheeger constant for $\mathcal{F}$, as desired.
\end{proof}

We conclude this section with a quick word on the differences between the proofs of Theorems \ref{thm:localplanarhypversion} and \ref{thm:localplanarhypversion-g5}: first, instead of adding edges $uv$ between consecutive vertices along $\eta$ to ensure $G_\Delta[D]$ is connected, we add new vertices $x_1,x_2$ to $D$ and paths $ux_1x_2v$. Call the set $D$ together with all new vertices $D'$. After performing this operation, $G_\Delta[D']$ is connected, and $|D'| \leq 3|D|$. Later, we invoke Theorems \ref{d_gepsilonboundg5} and \ref{expmanyextensions5} instead of the analogous Theorems \ref{d_gepsilonbound} and \ref{expmanyextensions} used in the girth three case. Finally, in the closing arguments where we count the number of vertices in $\Delta$ and invoke Euler's formula, we use the fact that the boundary walk of each face is at least five. This follows from the fact that either $D'$ has at least four vertices and $G$ has girth at least five, or $|D'| = |D| \leq 2$ and so since $G$ has girth at least five and does not contain $\frac{1}{25380}$-exponentially-3-correspondence-deletable subgraph, there are at least three vertices in $G_\Delta - D'$.   


\section{Further Directions}\label{sec:conclusion}


It is natural to wonder whether deficiency hyperbolicity theorems exist for other planar graph classes. (As discussed in Subsection \ref{subsec:hyperbolicity}, such theorems would have other interesting implications beyond the scope of counting colourings.) For instance, in \cite{esrlukelocal}, we introduce the concept of a \emph{local girth list assignment}. This is defined below, following another necessary definition.

The \emph{girth of an edge} and \emph{girth of a vertex} are defined as follows.

\begin{definition}
Let $G$ be a graph, and $e\in E(G)$.  The \emph{girth of $e$} is denoted by $g(e)$ and is defined as the length of the shortest cycle in $G$ in which $e$ is contained.  Similarly, for $v \in V(G)$, the \emph{girth of $v$} is denoted by $g(e)$ and is defined as the length of the shortest cycle in $G$ in which $v$ is contained. If a vertex or edge is not contained in a cycle, its girth is defined as being infinite.
\end{definition}

\begin{definition}
Let $G$ be a graph with list assignment $L$. We say $L$ is a \emph{local girth list assignment} if every vertex $v \in V(G)$ has $|L(v)| \geq 3$; if every vertex $v \in V(G)$ with $g(v) =4$ has $|L(v)| \geq 4$; and if every vertex $v \in V(G)$ with $g(v)=3$ has $|L(v)| \geq 5$. We say $G$ is \emph{local girth choosable} if $G$ has an $L$-colouring for every local girth list assignment $L$. \end{definition}

In \cite{esrlukelocal} we prove the following theorem. 

\begin{thm}\label{localgirth}
Every planar graph is local girth choosable.
\end{thm}
In light of the existing literature concerning exponentially many list colourings of planar graphs, it is natural to wonder, for each planar graph $G$ with local girth list assignment $L$, whether $G$ has exponentially many distinct $L$-colourings. Using Theorem \ref{thm:alonfuredi}, it is not too difficult to show this. Indeed, we prove this below, following a useful proposition.

 \begin{prop}\label{prop:eulergirth}
If $G$ is a planar graph that contains a cycle, then 
\[
v(G) - \sum_{e \in E(G)} \left(1-\frac{2}{g(e)} \right) \geq 2,
\]
where we interpret $\frac{2}{g(e)}$ to be $0$ if $g(e)=\infty$.
 \end{prop}
\begin{proof}
    We assume that $G$ is a plane graph (that is, we fix a planar embedding of $G$). Note that since $G$ contains at least one cycle, it follows that every face of $G$ contains at least one cycle in its boundary walk. By Euler's formula for graphs embedded in the plane, we have that $v(G) - e(G) + f(G) \geq 2$, where $f(G)$ denotes the number of faces of $G$. Now we assign a charge of +1 to each face and vertex of $G$, and -1 to each edge of $G$. Hence from above, the total sum of all charges is at least $+2$. 
    
    We discharge according to the following rule: for each face $f$, choose a cycle $C$ in its boundary walk, and send $\frac{1}{e(C)}$ charge to each edge in $C$. Note then that if an edge $e$ is contained in a cycle, it receives at most $+\frac{2}{g(e)}$ charge. If $e$ is not contained in a cycle, then since $g(e)$ is defined as infinite, again $e$ receives at most $\frac{2}{g(e)}$ charge (where $\frac{2}{g(e)} = 0$). Thus after discharging, the sum of the charges is at most $v(G) -e(G) + \sum_{e \in E(G)} \frac{2}{g(e)}$, and so

    \begin{align*}
         v(G) - \sum_{e \in E(G)} \left( 1 - \frac{2}{g(e)} \right) \geq 2, 
    \end{align*}
as desired.
\end{proof}

\begin{thm}
If $G$ is a planar graph and $L$ is a local girth list assignment for $G$, then $G$ has at least $5^\frac{v(G)}{12}$ distinct $L$-colourings.
\end{thm}
\begin{proof} 
We may assume without loss of generality that $|L(v)| = 5$ for all $v \in V(G)$ with $g(v) =3$; that $|L(v)| = 4$ for all $v \in V(G)$ with $g(v) = 4$; and that $|L(v)| = 3$ for all $v \in V(G)$ with $g(v) \geq 5$. We follow the polynomial method described in Subsection \ref{subsec:polmethod}.

First assume $G$ is a forest, in which case $e(G) \leq v(G)-1$. In this case, $|L(v)| \geq 3$ for all $v \in V(G)$, and so since $G$ has an $L$-colouring by Theorem \ref{localgirth}, it follows by the Alon-F\"uredi theorem (Theorem \ref{thm:alonfuredi}) that $G$ has at least $3^{\frac{3\cdot v(G) - v(G) -(v(G)-1)}{2}} = 3^\frac{v(G)+1}{2}$, as desired.

Thus we may assume $G$ is not a forest, and therefore contains a cycle.  Since there exists an $L$-colouring of $G$ by Theorem \ref{localgirth}, a direct application of the Alon-F\"uredi theorem (Theorem \ref{thm:alonfuredi}) gives that there are at least $5^\frac{S - v(G) -e(G)}{4}$ distinct $L$-colourings, where $S = \sum_{v \in V(G)} |L(v)|$.  For each $i \in \{3,4\}$, let $e_i := |\{e \in E(G): g(e) = i\}|$. Let $e_5 := |\{e \in E(G): g(e) \geq 5\}|$. Analogously, for each $i \in \{3,4\}$, let $v_i := |\{v \in V(G): g(v) = i\}|$, and let $v_5 := |\{v \in V(G): g(v) \geq 5\}|$. Since $L$ is a local girth list assignment, it follows that $S \geq 5v_3 + 4v_4 + 3v_5$, and so from above, we find that there exist at least $5^\frac{4v_3+3v_4+2v_5-e_3-e_4-e_5}{4}$ distinct $L$-colourings of $G$. It remains to bound the numerator in the exponent.

To that end, we define two subgraphs of $G$ as follows: let $G_3$ be the graph with $V(G_3):= \{v \in V(G): g(v) = 3\}$ and $E(G_3) := \{e \in E(G): g(e) = 3\}$. Let $G_4$ be the graph with $V(G_4):= \{v \in V(G): g(v) \in \{3,4\}\}$ and $E(G_4) = \{e \in E(G): g(e) \in \{3,4\}\}$.

Note that $v(G_3) = v_3$. Hence by Proposition \ref{prop:eulergirth} applied to $G_3$, we find that
\begin{equation}\label{eq:1}
    v_3 \geq \frac{e_3}{3}.
\end{equation}
Similarly, note that $v(G_4) = v_3 + v_4$. By Proposition \ref{prop:eulergirth} applied to $G_4$, we find that
\begin{equation}\label{eq:2}
    v_3 + v_4 \geq \frac{e_3}{3} + \frac{e_4}{2}.
\end{equation}
Finally, by Proposition \ref{prop:eulergirth} applied to $G$, we find that 
\begin{equation}\label{eq:3}
    v_3 + v_4 + v_5 \geq \frac{e_3}{3} + \frac{e_4}{2} + \frac{3e_5}{5}.
\end{equation}
Adding Equation \ref{eq:1} with $\frac{1}{3}$ of Equation \ref{eq:2} and $\frac{5}{3}$ of Equation \ref{eq:3}, we obtain
\begin{align*}
    3v_3 +2v_4 + \frac{5}{3} v_5 \geq e_3 + e_4 + e_5.
\end{align*}
Plugging this into our Alon-F\"uredi bound, we obtain the desired result.
\end{proof}

It is also natural to wonder whether an analogous result holds in the correspondence colouring framework. In this framework, it is not clear how one would define the polynomial required to invoke Theorem \ref{thm:alonfuredi}. A more pressing issue is that it is not currently known whether planar graphs are local girth correspondence colourable at all. However, we conjecture this is true.

\begin{conj}
   If $G$ is a planar graph and $(L,M)$ is a correspondence assignment for $G$ where $L$ is a local girth list assignment, then $G$ is $(L,M)$-colourable. 
\end{conj}

Our method (using deficiency versions of hyperbolicity theorems) could also be used to show planar graphs have exponentially many local girth correspondence colourings, but proving such a hyperbolicity theorem seems difficult. Indeed, it is not yet known whether there exists a hyperbolicity theorem for local girth list colouring. We conjecture this (and its correspondence colouring analogue) is true. 

\begin{conj}
    There exists a theorem analogous to Theorem \ref{theorem:mainhypthm} for local girth list colouring.
\end{conj}

\begin{conj}
    There exists a theorem analogous to Theorem \ref{theorem:mainhypthm} for local girth correspondence colouring.
\end{conj}
One might also wonder whether there exists a \emph{strong} hyperbolicity theorem for local girth list colouring; however, given the intricacy of the proof of even Theorem \ref{localgirth}, proving such a theorem seems a daunting task.

Returning our attentions to other surfaces:
in \cite{thomassen2006number}, Thomassen proved the following for arbitrary embedded graphs.
\begin{thm}[Thomassen, \cite{thomassen2006number}]\label{thm:thomassenexpmany5-colsurfaces}
For every surface $\Sigma$, there exists a constant $c > 0$ such that if $(G, \Sigma)$ is an embedded graph and $G$ is 5-colourable, then $G$ has at least $c \cdot 2^n$ distinct 5-colourings.
\end{thm}

In \cite{postle20213} and \cite{postle2018hyperbolic}, the first author and Thomas generalized Thomassen's results to list colouring as follows (in fact, they prove a stronger notion again involving counting the number of extensions of a precoloured subgraph).
\begin{thm}[Postle \& Thomas, \cite{postle2018hyperbolic}]
There exist constants $\epsilon, \alpha > 0$ such that the following holds. Let $G$ be a graph with $n$ vertices embedded in a fixed surface $\Sigma$ of genus $g$, and let $H$ be a proper subgraph of $G$. Suppose that either $L$ is a 5-list-assignment for $G$, or that $G$ has girth at least five and that $L$ is a 3-list-assignment for $G$. If $\phi$ is an $L$-colouring of $H$ that extends to an $L$-colouring of $G$, then $\phi$ extends to at least $2^{\varepsilon(n-\alpha(g+v(H)))}$ distinct $L$-colourings of $G$.   
\end{thm}
In \cite{kelly2018exponentially}, Kelly and Postle proved an analogous theorem for 4-list colouring embedded graphs of girth at least four. We note that these theorems  would also follow directly from the polynomial method described in Subsection \ref{subsec:polmethod}. Using the approach of Dahlberg, Kaul, and Mudrock \cite{dahlberg2023algebraic}, analogous results for 3-correspondence colouring embedded graphs of girth at least five also hold. Again, the polynomial method does not seem amenable to proving there are exponentially many 5-correspondence colourings of embedded graphs, though this should follow if one could prove \emph{strong} hyperbolicity theorems for this class of graphs.  The first author and Thomas proved such a theorem for 5-list colouring; but as the full proof is over 200 pages, generalizing such a result to the correspondence colouring framework seems a formidable task. 

\paragraph{Acknowledgement.}
Several results in this paper form part of the doctoral dissertation \cite{evethesis} of the second author, written under the guidance of the first.

\bibliographystyle{siam}
\bibliography{bibliog}
\end{document}